\documentclass[11pt]{amsart}
\usepackage{amscd}
\usepackage[all]{xy}
\usepackage{amsmath}
\usepackage{amsfonts}
\usepackage{amssymb}
\usepackage{latexsym}
\usepackage{extarrows}
\usepackage{mathtools}

\numberwithin{equation}{section}
\theoremstyle{plain}
\newtheorem{lemma}{Lemma}[section]
\newtheorem{proposition}[lemma]{Proposition}
\newtheorem{theorem}[lemma]{Theorem}
\newtheorem{corollary}[lemma]{Corollary}

\theoremstyle{definition}
\newtheorem{definition}[lemma]{Definition}

\newtheorem{example}[lemma]{Example}

\newcommand{\R}{{\mathbb R}}
\newcommand{\Z}{{\mathbb Z}} 
\newcommand{\N}{{\mathbb N}}
\newcommand{\Hh}{{\mathcal H}}
\newcommand{\Ll}{{\mathcal L}}
\newcommand{\ad}{{\rm ad}}
\newcommand{\Om}{{\Omega}}
\newcommand{\om}{{\omega}}
\newcommand{\X}{{\mathfrak X}}
\newcommand{\la}{\langle}
\newcommand{\ra}{\rangle}
\newcommand{\p}{{\partial}}
\newcommand{\vol}{\mbox{\rm vol}}
\newcommand{\rmspan}{\mbox{\rm span}}
\newcommand{\codim}{\mbox{\rm codim}}
\newcommand{\Sp}{{\text{\rm Spin}(7)}}
\renewcommand{\Im}{{ \rm Im \,}}
\renewcommand{\frak}[1]{{\mathfrak {#1}}}

\begin{document}
\date{\today}
\title[Fr\"olicher-Nijenhuis cohomology on $G_2$- and ${\rm Spin}(7)$-manifolds]
{Fr\"olicher-Nijenhuis cohomology on $G_2$- and ${\rm Spin}(7)$-manifolds}

\author{Kotaro Kawai}
\address{Graduate School of Mathematical Sciences, University of Tokyo, 3-8-1, Komaba, Meguro, Tokyo 153-8914, Japan}
\email{kkawai@ms.u-tokyo.ac.jp}

\author{H\^ong V\^an L\^e}
\address{Institute of Mathematics CAS
, Zitna 25, 11567 Praha 1, Czech Republic}
\email{hvle@math.cas.cz}

\author{Lorenz Schwachh\"ofer}
\address{Fakult\"at f\"ur Mathematik,
TU Dortmund University,
Vogelpothsweg 87, 44221 Dortmund, Germany}
\email{lschwach@math.tu-dortmund.de} 

\thanks{The first named author is supported Grant-in-Aid for JSPS fellows (26-7067),
the second named author was supported by RVO: 67985840 and the GA\v CR-project 18-00496S. The third author acknowledges partial support by grant SCHW893/5-1 of the Deutsche Forschungsgemeinschaft.}

\begin{abstract} In this paper we show that a parallel differential form $\Psi$ of even degree on a Riemannian manifold allows to define a natural differential both on $\Om^\ast(M)$ and $\Om^\ast(M, TM)$, defined via the Fr\"olicher-Nijenhuis  bracket. For instance, on a K\"ahler manifold, these operators are the complex differential and the Dolbeault differential, respectively. We investigate this construction when taking the differential w.r.t. the canonical parallel $4$-form on a $G_2$- and $\Sp$-manifold, respectively. We calculate the cohomology groups of $\Om^\ast(M)$ and give a partial description of the cohomology of $\Om^\ast(M, TM)$.
\end{abstract}

\maketitle

\section{Introduction}

In Riemannian geometry, the investigation of manifolds with special holonomy is a central field of research. Assuming a Riemannian manifold is irreducible and not locally symmetric, there are only few possible subgroups of the orthogonal group which can occur as holonomies, classified by Berger \cite{Ber}. Most of the possible holonomy groups may be characterized as the stabilizer of (one or two) differential forms in the orthogonal group, leading to the existence of parallel forms on such manifolds. For instance, K\"ahler manifolds admit the parallel K\"ahler-$2$-form, Calabi-Yau manifolds admit in addition a complex volume form \cite{Yau-3}; hyper-K\"ahler manifolds admit three linearly independent parallel $2$-forms \cite{Cal5}, and similarly, manifolds with exceptional holonomy $G_2$ in dimension $7$ and $\Sp$ in dimension $8$ may be characterized by a parallel $4$-form \cite{Bonan}, \cite{Gra1}.  For the extremely  rich structure of manifolds with these holonomy groups, we refer to \cite{Bryant1987}, \cite{Bryant2005}, \cite{BS1989}, \cite{HL1982}, and for the discussion of closed manifolds with these holonomies we refer to \cite{JoyceG2}, \cite{JoyceSp}, \cite{Joyce2000} and \cite{Joyce2007}. We
also  refer to \cite{Le2013} for description  of Riemannian  manifolds  admitting a parallel 3-form.

It is the aim of the present paper to construct cohomology invariants on an oriented Riemannian manifold $(M,g)$ with a parallel form of even degree. Namely, to such a form, say $\Psi^{2k}$, we associate two differentials
\[
\Ll_{\Psi^{2k}}: \Om^\ast(M) \to \Om^{\ast+2k-1}(M), \qquad \Ll_{\Psi^{2k}}: \Om^\ast(M, TM) \to \Om^{\ast+2k-1}(M, TM),
\]
i.e., $\Ll_{\Psi^{2k}}$ is a derivation   and $\Ll_{\Psi^{2k}} \circ \Ll_{\Psi^{2k}} = 0$ in each case. Their cohomologies, denoted as $H^\ast_{\Psi^{2k}}(M)$ and $H^\ast_{\Psi^{2k}}(M, TM)$, respectively, form a graded algebra and a graded Lie algebra, respectively, where the $i$-th cohomologies are defined as
\begin{equation} \label{eq:def-cohom}
\begin{array}{rcl}H^i_{\Psi^{2k}}(M) & := & \dfrac{\ker \Ll_{\Psi^{2k}}: \Om^i(M) \to \Om^{i+2k-1}(M)}{\Im \Ll_{\Psi^{2k}}: \Om^{i-2k+1}(M) \to \Om^i(M)},\\[5mm] H^i_{\Psi^{2k}}(M, TM) & := & \dfrac{\ker \Ll_{\Psi^{2k}}: \Om^i(M, TM) \to \Om^{i+2k-1}(M, TM)}{\Im \Ll_{\Psi^{2k}}: \Om^{i-2k+1}(M, TM) \to \Om^i(M, TM)}.\end{array}
\end{equation}

The action of $\Ll_{\Psi^{2k}}$ on differential forms is given by the formula
\[
\Ll_{\Psi^{2k}}(\alpha) = (d^\ast \alpha) \wedge \Psi^{2k} - d^\ast (\alpha \wedge \Psi^{2k}).
\]

For instance, in the case of a K\"ahler manifold, using the K\"ahler form $\Psi = \om$, the differential $\Ll_\om$ on $\Om^\ast(M)$ is the complex differential $d^c := i (\p- \bar{\p})$, whereas on $\Om^\ast(M, TM)$ it coincides with the Dolbeault differential $\bar{\p}: \Om^{p,q}(M, TM) \to \Om^{p,q+1}(M, TM)$ \cite{FN1956b}, cf. Example \ref{ex:Kaehler}. Thus, these differentials recover well known and natural cohomology theories.

On the other hand, for $G_2$- and $\Sp$-manifolds, there are canonical parallel $4$-forms, denoted by $\ast \varphi$ and $\Phi$, respectively, and we may consider the respective differentials $\Ll_{\ast \varphi}$ and $\Ll_\Phi$. On closed manifolds, we obtain the following results on their cohomology groups.

\begin{theorem}\label{thm:hG2} Let $(M^7, \varphi)$ be a closed $G_2$-manifold. Then for the cohomologies $H^i_{\ast \varphi}(M^7)$ and $H^i_{\ast \varphi}(M^7, TM^7)$ defined above, the following hold.
\begin{enumerate}
\item There is a Hodge decomposition
\begin{align*}
H^i_{\ast \varphi}(M^7) =\;& \Hh^i(M^7) \oplus (H^i_{\ast \varphi}(M^7) \cap d\Om^{i-1}(M^7)) \\ &\oplus (H^i_{\ast \varphi}(M^7) \cap d^\ast\Om^{i+1}(M^7)),
\end{align*}
where $\Hh^i(M^7)$ denotes the spaces of harmonic forms.
\item The Hodge-$\ast$ induces an isomorphism $\ast: H^i_{\ast \varphi}(M^7) \to H^{7-i}_{\ast \varphi}(M^7)$.
\item $H^i_{\ast \varphi}(M^7) = \Hh^i(M^7)$ for $i = 0,1,6,7$. For $i = 2,3,4,5$, $H^i_{\ast \varphi}(M^7)$ is infinite dimensional.
\item $\dim H^0_{\ast \varphi}(M^7, TM^7) = b^1(M^7)$; in particular, $H^0_{\ast \varphi}(M^7) = 0$ if $M^7$ has full holonomy $G_2$.
\item $\dim H^3_{\ast \varphi}(M^7, TM^7) \geq b^3(M^7) > 0$, as it contains all torsion free deformations of the $G_2$-structure modulo deformations by diffeomorphisms.
\end{enumerate}
\end{theorem}

In fact, in Theorem \ref{thm:cohomG2} we shall give a precise description of the cohomology ring $H^\ast_{\ast \varphi}(M^7)$. Analogously, for $\Sp$-manifolds we have

\begin{theorem}\label{thm:hSp} Let $(M^8, \Phi)$ be a closed $\Sp$-manifold. Then for the cohomologies $H^i_{\Phi}(M^8)$ and $H^i_{\Phi}(M^8, TM^8)$ defined above, the following hold.
\begin{enumerate}
\item There is a Hodge decomposition
\begin{align*}
H^i_{\Phi}(M^8) =\;& \Hh^i(M^8) \oplus (H^i_{\Phi}(M^8) \cap d\Om^{i-1}(M^8))\\ & \oplus (H^i_{\Phi}(M^8) \cap d^\ast\Om^{i+1}(M^8)),
\end{align*}
where $\Hh^i(M^8)$ denotes the spaces of harmonic forms.
\item The Hodge-$\ast$ induces an isomorphism $\ast: H^i_{\Phi}(M^8) \to H^{8-i}_{\Phi}(M^8)$.
\item $H^i_{\Phi}(M^8) = \Hh^i(M^8)$ for $i = 0,1,2,6,7,8$. For $i = 3,4,5$, $H^i_\Phi(M^8)$ is infinite dimensional.
\item $\dim H^0_\Phi(M^8, TM^8) = b^1(M^8)$; in particular, $H^0_\Phi(M^8) = 0$ if $M^8$ has full holonomy $\Sp$.
\item $\dim H^3_\Phi(M^8, TM^8) \geq b^4_1(M^8) + b^4_7(M^8) + b^4_{35}(M^8) > 0$, as it contains all torsion free deformations of the $\Sp$-structure modulo deformations by diffeomorphisms. Here, $b^4_k(M)$ stands for the dimension of the space of harmonic forms in the parallel rank $k$-subbundle $\Lambda^4_k T^\ast M \subset \Lambda^4 T^\ast M$; cf. (\ref{eq:bkl}).
\end{enumerate}
\end{theorem}

In fact, in Theorem \ref{thm:cohomSpin7} we shall give a precise description of the cohomology ring $H^\ast_\Phi(M^8)$.

In order to define these cohomologies, we utilize the {\em Fr\"olicher-Nijenhuis bracket $[\;,\; ]^{FN}$} which provides the graded space $\Om^\ast(M, TM)$ with the structure of a graded Lie algebra, acting on the graded algebra $\Om^\ast(M)$ by graded derivations, say, $\Ll_K: \Om^\ast(M) \to \Om^\ast(M)$ for $K \in \Om^\ast(M, TM)$, and on $\Om^\ast(M, TM)$ by the adjoint $ad_K$ \cite{FN1956}. If $\Psi^{2k}$ is a parallel form and $\hat \Psi := \p_g \Psi^{2k} \in \Om^{2k-1}(M, TM)$ is the contraction with the metric $g$, then $[\hat \Psi^{2k}, \hat \Psi^{2k}]^{FN} = 0$, which implies that $\Ll_{\hat \Psi^{2k}}$ (which is also denoted by $\Ll_{\Psi^{2k}}$) is a differential, so its cohomology can be defined as above.

If the underlying manifold is closed, then this cohomology has a Hodge decomposition and contains all harmonic forms, as will be shown in Theorems \ref{thm:hG2} and \ref{thm:hSp}. In fact, in this case the cohomology algebra $(H^\ast_\Psi, d)$ is {\em formally equivalent }to the deRham algebra $(\Om^\ast(M), d)$ (cf. Corollary \ref{cor:formal}) and hence by \cite{Sullivan} contains all information on the rational homotopy type of $M$. This is one of our motivations for studying this cohomology in detail.

If the differential $\Ll_\Psi$ satisfies some regularity condition, then we show that the Hodge-duality of the cohomology holds. As it turns out, this regularity condition is satisfied for the parallel $4$-forms on $G_2$- and $\Sp$-manifolds, respectively, allowing us to give the explicit descriptions of the cohomology in Theorems \ref{thm:hG2} and \ref{thm:hSp}.

Our paper is organized as follows. In Section \ref{sec:pre} we provide the basic definitions of graded Lie algebras and the Fr\"olicher-Nijenhuis bracket, and show that a parallel form of even degree on a Riemannian manifold yields a differential on both $\Om^\ast(M)$ and $\Om^\ast(M, TM)$ which has the properties asserted above for closed $M$. In Section \ref{sec:g2} we apply this to $G_2$-manifolds, giving explicit descriptions of the cohomology $H^\ast_{\ast \varphi}(\Om^\ast(M))$ and repeat this in Section \ref{sec:spin7} for $\Sp$-manifolds. In Section \ref{sec:deform} we discuss the relation of the cohomology groups $H^\ast(M, TM)$ and deformations, and show the results on the cohomologies $H^\ast_{\ast \varphi}(M^7, TM^7)$ and $H^\ast_\Phi(M^8, TM^8)$, respectively.
Finally, we outsourced some technical results into the appendix.

\section{Preliminaries}\label{sec:pre}
\subsection{The Hodge-$\ast$ operator}

Let $(V, g)$ be an $n$-dimensional oriented real vector space with a scalar product $g$ and volume form $\vol$. Then the {\em musical operators} allow us to identify $V$ and $V^\ast$ by
\[
\begin{array}{rrl}& \flat: & V \longrightarrow V^\ast, \qquad v \longmapsto v^\flat := g(v,\cdot)\\
\mbox{and}\\ & \#:= \flat^{-1}: & V^\ast \longrightarrow V, \qquad \alpha \longmapsto \alpha^\#.
\end{array}
\]

$g$ also induces an inner product on $\Lambda^k V^\ast$ in a canonical way, and the {\em Hodge-$\ast$ operator $\ast: \Lambda^k V^\ast \to \Lambda^{n-k} V^\ast$ }is the unique linear map satisfying the identity
\[
\alpha^k \wedge \beta^{n-k} = g(\ast \alpha^k, \beta^{n-k})\; \vol = g(\ast \alpha^k, \beta^{n-k}) \ast 1
\]
for all $\alpha^k \in \Lambda^k V^\ast$ and $\beta^{n-k} \in \Lambda^{n-k} V^\ast$. In particular,
\begin{equation} \label{prop-Hodge}
\ast^2|_{\Lambda^k V^\ast} = (-1)^{k(n-k)} Id_{\Lambda^k V^\ast} \qquad \mbox{and} \qquad g(\ast \alpha^k, \ast \beta^k) = g(\alpha^k, \beta^k).
\end{equation}
Moreover, for $\alpha^k \in \Lambda^k V^\ast$ and $v \in V$, we have the relation
\begin{equation} \label{eq:ast-contr}
\imath_v (\ast \alpha^k) = (-1)^k \ast (v^\flat \wedge \alpha^k) \qquad \mbox{and} \qquad \ast( \imath_v \alpha^k) = (-1)^{k+1} v^\flat \wedge \ast \alpha^k,
\end{equation} 
where $\imath_v$ denotes the insertion of the vector $v$ into the form $\alpha^k$.

Let $(M,g)$ be an $n$-dimensional oriented Riemannian manifold and let $\nabla$ denotes its Levi-Civita connection. Then $\nabla$ is compatible with $\ast$, i.e.,
\[
\ast (\nabla_v \alpha^k) = \nabla_v (\ast \alpha^k).
\]

A (local) oriented orthonormal frame $(e_i)$ with dual (local) coframe $(e^i) = (e_i^\flat)$ is called {\em normal in $p \in M$ }if $\nabla e_i|_p = 0$ or, equivalently, $\nabla e^i|_p = 0$ for all $i$. The existence of such a (co-)frame can be shown e.g. by choosing normal coordinates $(x^i)$ around $p$, so that $p \cong 0$ in these coordinates, and letting $e_i := {\sqrt g}^{ij} \p_j$, where $\p_j := \p/\p x_j$ denotes the coordinate vector field. If $(e_i)$ is normal in $p$ then the exterior derivative $d$ and its formal dual $d^\ast$ at $p$ are given as
\begin{equation} \label{d*-formula}
d\alpha^k|_p = e^i \wedge \nabla_{e_i} \alpha^k, \qquad
d^\ast\alpha^k|_p = (-1)^{n(n-k)+1} \ast d \ast \alpha^k = - \imath_{e_i} \nabla_{e_i} \alpha^k.
\end{equation}

Finally, for later reference, we also recall the map defined by contraction of a form with the metric $g$, i.e.
\begin{equation} \label{eq:def-partial}
\p = \p_g: \Lambda^k V^\ast \longrightarrow \Lambda^{k-1} V^\ast \otimes V, \qquad \p_g(\alpha^k) := (\imath_{e_i} \alpha^k) \otimes (e^i)^\#,
\end{equation}
where the sum is taken over some basis $(e_i)$ of $V$ with dual basis $(e^i)$.

\subsection{Graded Lie algebras and differentials}\label{sec:coh}

In this section, we briefly recall some basic notions and properties of graded algebras and graded Lie algebras. Let $V := (\bigoplus_{k \in \Z} V_k, \cdot)$ be a graded real vector space with a graded bilinear map $\cdot: V \times V \to V$, called a {\em product on $V$}. A {\em graded derivation of $(V, \cdot)$ of degree $l$ }
is a linear map $D^l: V \to V$ of degree $l$ (i.e., $D^l (V_k) \subset V_{k+l}$) such that 
\begin{equation} \label{eq:derivations}
D^l(x \cdot y) = (D^l x) \cdot y + (-1)^{l|x|} x \cdot (D^l y),
\end{equation}
where $|x|$ denotes the degree of an element, i.e. $|x| = k$ for $x \in V_k$.
If we denote by ${\mathcal D}^l(V)$ the graded derivations of $(V, \cdot)$ of degree $l$, then ${\mathcal D}(V) := \bigoplus_{l \in \Z} {\mathcal D}^l(V)$
 is a graded Lie algebra with the Lie bracket
\begin{equation} \label{eq:derivations-bracket}
{}[D_1, D_2] := D_1 D_2 - (-1)^{|D_1| |D_2|} D_2 D_1,
\end{equation}
i.e., the Lie bracket is graded anti-symmetric and satisfies the graded Jacobi identity,
\begin{eqnarray}
\label{GLA-skew} && [x, y] = -(-1)^{|x||y|} [y,x]\\
\label{GLA-Jac} && (-1)^{|x||z|} [x, [y,z]] + (-1)^{|y||x|} [y, [z,x]] + (-1)^{|z||y|} [z, [x,y]] = 0.
\end{eqnarray}

In general, if $L = (\bigoplus_{k \in \Z} L_k, [\cdot, \cdot])$ is a graded Lie algebra, then an {\em action of $L$ on $V$ }is a Lie algebra homomorphism $\pi: L \to {\mathcal D}(V)$, which yields a graded bilinear map $L \times V \to V$, $(x, v) \mapsto \pi(x)(v)$ such that the map $\pi(x): V \to V$ is a graded derivation of degree $|x|$ and such that
\[
{}[\pi(x), \pi(y)] = \pi[x,y].
\]
For instance, a graded Lie algebra acts on itself via the adjoint representation $ad: L \to {\mathcal D}(L)$, where $ad_x(y) := [x,y]$.

For a graded Lie algebra $L$ we define the set of {\em Maurer-Cartan elements of $L$ of degree $2k+1$ }as
\[
{\mathcal {MC}}^{2k+1}(L) := \{ \xi \in L_{2k+1} \mid [\xi,\xi] = 0\}.
\]
If $\pi: L \to {\mathcal D}(V)$ is an action of $L$ on $(V, \cdot)$, then for $\xi \in {\mathcal {MC}}^{2k+1}(L)$ we have $0 = [\pi(\xi), \pi(\xi)] = 2 \pi(\xi)^2$, so that $\pi(\xi): V \to V$ is a differential on $V$. We define the {\em cohomology of $(V, \cdot)$ w.r.t. $\xi$ }as
\begin{equation} \label{eq:cohom-GA}
H^i_\xi(V): = 
\frac{\ker (\pi(\xi): V_i \to V_{i+2k+1})}{\Im (\pi(\xi): V_{i-(2k+1)} \to V_i)} \qquad \mbox{for $\xi \in {\mathcal {MC}}^{2k+1}(L)$.}
\end{equation}
Since $\pi(\xi)$ is a derivation, it follows that $\ker \pi(\xi) \cdot \ker \pi(\xi) \subset \ker \pi(\xi)$, whence there is an induced product on $H^\ast_\xi(V) := \bigoplus_{i \in \Z} H^i_\xi(V)$.

If $L = \bigoplus_{k \in \Z} L_k$ is a graded Lie algebra, then for $v \in L_0$ and $t \in \R$, we define the formal power series
\begin{equation} \label{eq:formal-exp}
\exp(tv): L \longrightarrow L[[t]], \qquad \exp(tv)(x) := \sum_{k=0}^\infty \dfrac{t^k}{k!} ad_v^k(x).
\end{equation}
Observe that $ad_\xi(v) = 0$ for some $v \in L_0$ iff $ad_v(\xi) = 0$ iff $\exp(tv)(\xi) = \xi$ for all $t \in \R$. In this case, we call $v$ an {\em infinitesimal stabilizer of $\xi$}.

For $\xi \in {\mathcal {MC}}^{2k+1}(L)$, we say that $x \in L_{2k+1}$ is an {\em infinitesimal deformation of $\xi$ within ${\mathcal {MC}}^{2k+1}(L)$ }if $[\xi + t x, \xi + t x] = 0 \mod t^2$. Evidently, this is equivalent to $[\xi, x] = 0$ or $x \in \ker ad_\xi$. Such an infinitesimal deformation is called {\em trivial} if $x = [\xi, v]$ for some $v \in L_0$, since in this case, $\xi + t x = \exp(-tv)(\xi) \mod t^2$, whence up to second order, it coincides with elements in the orbit of $\xi$ under the (formal) action of $\exp(tv)$.
Thus, we have the following interpretation of some cohomology groups.

\begin{proposition} \label{prop:homology-general}
Let $(L = \bigoplus_{i \in \Z} L_i, [\cdot, \cdot])$ be a real graded Lie algebra, acting on itself by the adjoint representation, and let $\xi \in {\mathcal {MC}}^{2k+1}(L)$. Then the following holds.
\begin{enumerate}
\item If $L_{-(2k+1)} = 0$, then $H_\xi^0(L)$ is the Lie algebra of infinitesimal stabilizers of $\xi$.
\item $H_\xi^{2k+1}(L)$ is the space of infinitesimal deformations of $\xi$ within \linebreak ${\mathcal {MC}}^{2k+1}(L)$ modulo trivial deformations.
\end{enumerate}
\end{proposition}

\subsection{The Fr\"olicher-Nijenhuis bracket}\label{subs:fnb}

We shall apply our discussion from the preceding section to the following example. Let $M$ be a manifold and $(\Om^\ast(M), \wedge) = (\bigoplus_{k \geq 0} \Om^k(M), \wedge)$ be the graded algebra of differential forms. Evidently, the exterior derivative $d: \Om^k(M) \to \Om^{k+1}(M)$ is a derivation of $\Om^\ast(M)$ of degree $1$, whereas insertion $\imath_X: \Om^k(M) \to \Om^{k-1}(M)$ of a vector field $X \in {\frak X}(M)$ is a derivation of degree $-1$.

More generally, for $K \in \Om^k(M, TM)$ we define $\imath_K \alpha^l$ as the {\em insertion of $K$ into $\alpha^l \in \Om^l(M)$ }pointwise by
\[
\imath_{\kappa^k \otimes X} \alpha^l := \kappa^k \wedge (\imath_X \alpha^l) \in \Om^{k+l-1}(M),
\]
where $\kappa^k \in \Om^k(M)$ and $X \in {\frak X}(M)$, and this is a derivation of $\Om^\ast(M)$ of degree $k-1$. Thus, the {\em Nijenhuis-Lie derivative along $K \in \Om^k(M, TM)$ }defined as
\begin{equation} \label{eq:LK-deriv}
\Ll_K (\alpha^l) := [\imath_K, d] (\alpha^l) = \imath_K (d\alpha^l) + (-1)^k d(\imath_K \alpha^l) \in \Om^{k+l}(M)
\end{equation}
is a derivation of $\Om^\ast(M)$ of degree $k$, and evidently,
\begin{equation} \label{eq:L-d-commute}
\Ll_K(d\alpha) = (-1)^k d\Ll_K(\alpha).
\end{equation}
Moreover, if $K = \kappa_i^k \otimes e_i$ w.r.t. some local frame $(e_i)$, then
\begin{equation} \label{eq:Lie-2}
\Ll_K (\alpha^l) := \kappa_i^k \wedge \Ll_{e_i} \alpha^l + (-1)^k  d\kappa_i^k \wedge \imath_{e_i} \alpha^l.
\end{equation}

Observe that for $k = 0$ in which case $K \in \Om^0(M, TM)$ is a vector field, both $\imath_K$ and $\Ll_K$ coincide with the standard notion of insertion of and Lie derivative along a vector field.

In \cite{FN1956} \cite{FN1956b}, it was shown that $\Om^\ast(M, TM)$ carries a unique structure of a graded Lie algebra, defined by the so-called {\em Fr\"olicher-Nijenhuis bracket},
\[
[\cdot, \cdot]^{FN}:  \Om^k (M, TM) \times \Om^l (M, TM) \to \Om^{k+l} (M, TM)
\]
such that $\Ll$ defines an action of $\Om^\ast(M, TM)$ on $\Om^\ast(M)$, that is,
\begin{equation} \label{eq:FN-homom}
\Ll_{[K_1, K_2]^{FN}} = [\Ll_{K_1}, \Ll_{K_2}] = \Ll_{K_1} \circ \Ll_{K_2} - (-1)^{|K_1||K_2|} \Ll_{K_2} \circ \Ll_{K_1}.
\end{equation}
It is given by the following formula for $\alpha^k \in \Om^k(M)$, $\beta^l \in\Om ^l (M)$, $X_1, X_2 \in \X (M)$ \cite[Theorem 8.7 (6), p. 70]{KMS1993}:
\begin{align}
\nonumber [\alpha^k \otimes  X_1, &\beta^l \otimes  X_2]^{FN} = \alpha^k \wedge \beta^l \otimes [ X_1,  X_2]\nonumber \\
& + \alpha^k \wedge \Ll_{X_1} \beta^l \otimes X_2 
- \Ll_{X_2} \alpha^k \wedge \beta^l \otimes X_1 \label{eq:kms}\\ 
&+ (-1)^{k} \left( d \alpha^k \wedge (\imath_{X_1} \beta^l) \otimes X_2 
+ (\imath_{X_2} \alpha^k) \wedge d \beta^l \otimes X_1 \right).
\nonumber
\end{align}
In particular, for a vector field $X \in  \X(M)$ and $K \in \Om^\ast(M, TM)$ we have \cite[Theorem 8.16 (5), p. 75]{KMS1993}
\begin{align*}
\Ll_X (K) = [X, K] ^{FN},
\end{align*}
that is, the Fr\"olicher-Nijenhuis bracket with a vector field coincides with the Lie derivative of the tensor field $K \in \Om^\ast(M, TM)$. This means that $\exp(tX): \Om^\ast(M, TM) \to \Om^\ast(M, TM)[[t]]$ is the action induced by (local) diffeomorphisms of $M$. Thus, Proposition \ref{prop:homology-general} now immediately implies the following result.

\begin{theorem} \label{thm:FN-cohomology}
Let $M$ be a manifold and $K \in \Om^{2k+1}(M, TM)$ be such that $[K,K]^{FN}= 0$, and define the differential $d_K(K') := [K, K']^{FN}$. Then
\begin{enumerate}
\item $H_K^0(\Om^\ast(M, TM))$ is the Lie algebra of vector fields stabilizing $K$.
\item $H_K^{2k+1}(\Om^\ast(M, TM))$ is the space of infinitesimal deformations of $K$ within the differentials of $\Om^\ast(M, TM)$ of the form $ad_{\xi^{2k+1}}$, modulo (local) diffeomorphisms.
\end{enumerate}
\end{theorem}

\subsection{Riemannian manifolds with parallel forms} \label{subs:parallel}

Suppose that $(M,g)$ is an $n$-dimensional Riemannian manifold with Levi-Civita connection $\nabla$, and suppose that $\Psi$ is a parallel form of even degree. We now make the following simple but crucial observation.

\begin{proposition} \label{prop:FN}
Let $(M, g)$ be a Riemannian manifold and $\Psi \in \Om^{2k}(M)$ be a parallel form of even degree, and let $\hat \Psi := \p_g \Psi \in \Om^{2k-1}(M, TM)$ with the contraction map $\p_g$ from (\ref{eq:def-partial}).

Then $\hat \Psi$ is a Maurer-Cartan element, i.e., $[\hat \Psi, \hat \Psi]^{FN} = 0$.
\end{proposition}

\begin{proof} Choose geodesic normal coordinates $(x^i)$ around $p \in M$ in such a way that $(\p_i)_p := (\p / \p x^i)_p$ is an orthonormal basis of $T_pM$.
The dual basis of $\p_i$ is $dx^i$, whence $(dx^i)^\# = g^{ij} \p_j$. Thus,
\[
\hat \Psi = (\imath_{\p_i} \Psi) \otimes (dx^i)^\# = g^{ij} (\imath_{\p_i} \Psi) \otimes \p_j.
\]
Thus, by (\ref{eq:kms})
\begin{align*}
{}[\hat \Psi,  \hat \Psi]^{FN} = 
&[g^{ij} (\imath_{\p_i} \Psi) \otimes \p_j, g^{rs} (\imath_{\p_r} \Psi) \otimes \p_s]^{FN}\\
=  
&\left(
g^{ij} (\imath_{\p_i} \Psi) \wedge \Ll_{\p_j} (g^{rs} (\imath_{\p_r} \Psi))) \otimes \p_s \right.\\
&-
\Ll_{\p_s} (g^{ij} (\imath_{\p_i} \Psi)) \wedge g^{rs} (\imath_{\p_r} \Psi) \otimes \p_j \\
&- d (g^{ij} (\imath_{\p_i} \Psi)) \wedge \imath_{\p_j} (g^{rs} (\imath_{\p_r} \Psi)) \otimes \p_s \\
&- (\imath_{\p_s}(g^{ij} (\imath_{\p_i} \Psi)) \wedge d(g^{rs} (\imath_{\p_r} \Psi))) \otimes \p_j.
\end{align*}
Since at $p$, $g_{ij} = g^{ij} = \delta_{ij}, \p_r g_{ij} = 0$,
$\Ll_{\p_j} \Psi = \nabla_{e_j} \Psi = 0$, $\nabla_{\p_i} \p_j = 0$, and $\p_j = (e^j)^\#$, it follows that $[\hat \Psi, \hat \Psi]^{FN}_p = 0$, and $p \in M$ was arbitrary.
\end{proof}

Thus, by the discussion in Section \ref{sec:coh}, the Lie derivative $\Ll_{\hat \Psi}$ and the adjoint map $ad_{\hat \Psi}$ are differentials on $\Om^\ast(M)$ and $\Om^\ast(M, TM)$, respectively, and for simplicity, we shall denote these by
\[
\Ll_\Psi: \Om^\ast(M) \longrightarrow \Om^\ast(M) \qquad \mbox{and} \qquad 
ad_\Psi: \Om^\ast(M, TM) \longrightarrow \Om^\ast(M, TM),
\]
or, if we wish to specify the degree,
\begin{equation} \label{eq:definf-Lpsil}
\begin{array}{lllll}
\Ll_{\Psi; l}: & \Om^{l-2k+1}(M) & \longrightarrow & \Om^l(M) & \mbox{and}\\[2mm]
ad_{\Psi; l}: & \Om^{l-2k+1}(M, TM) & \longrightarrow & \Om^l(M, TM).
\end{array}
\end{equation}
The cohomology algebras we denote by $H_\Psi^\ast(M)$ and $H_\Psi^\ast(TM)$ instead of $H_{\hat \Psi}^\ast(\Om^\ast(M))$ and $H_{\hat \Psi}^\ast(\Om^\ast(M, TM))$, respectively. That is,
\begin{equation}
\label{eq:def-HKlM}
H_\Psi^l(M) = \dfrac{ \ker (\Ll_{\Psi; l + 2k-1})}{\Im (\Ll_{\Psi;l})} \quad \mbox{and} \quad 
H_\Psi^l(TM) = \dfrac{ \ker (\ad_{\Psi; l + 2k - 1})}{\Im (ad_{\Psi; l})}.
\end{equation}
Observe that a priori we do not have any well behaved topology on $H_\Psi^l(M)$ and $H_\Psi^l(TM)$, since it is not clear if the denominators in (\ref{eq:def-HKlM}) are closed subspaces of the respective numerator in the Frech\'et topology.

We assert that
\begin{equation}\label{eq:LK-d*}
\Ll_\Psi \alpha^l = (d^\ast\alpha^l) \wedge \Psi - d^\ast(\alpha^l \wedge \Psi).
\end{equation}
In order to see this, observe that just as in the proof of Proposition \ref{prop:FN} for an orthonormal frame $(e_i)$ which is normal at $p \in M$ (\ref{eq:Lie-2}) implies that
\begin{equation}\label{eq:LK-nabla}
\Ll_\Psi \alpha^l|_p = (\imath_{e_i} \Psi) \wedge \nabla_{e_i} \alpha^l|_p,
\end{equation}
and since $\nabla \Psi = 0$ we have
\begin{align*}
(\imath_{e_i} \Psi) \wedge \nabla_{e_i} \alpha^l|_p = \imath_{e_i} \nabla_{e_i} (\Psi \wedge \alpha^l)|_p - \Psi \wedge (\imath_{e_i} \nabla_{e_i} \alpha^l)|_p,
\end{align*}
from which (\ref{eq:LK-d*}) follows at $p \in M$ by (\ref{d*-formula}).
In particular, (\ref{eq:LK-d*}) together with (\ref{eq:L-d-commute}) implies
\begin{equation} \label{eq:LK-commutes}
\Ll_\Psi d\alpha^l = -d \Ll_\Psi \alpha^l \qquad \mbox{and} \qquad \Ll_\Psi d^\ast \alpha^l = - d^\ast \Ll_\Psi \alpha^l
\end{equation}
and therefore,
\begin{equation} \label{eq:LK-comm-Laplace}
\Ll_\Psi \triangle \alpha^l = \triangle \Ll_\Psi \alpha^l.
\end{equation}

By the Weitzenb\"ock formula (see e.g. \cite[(1.154)]{Besse1987}), it follows that exterior multiplication with a parallel form commutes with the Laplacian. That is, we have for all $\alpha \in \Om^\ast(M)$
\begin{equation} \label{eq:Laplace-commute}
\triangle (\alpha \wedge \Psi) = (\triangle \alpha) \wedge \Psi \qquad \mbox{and} \qquad \triangle (\alpha \wedge \ast \Psi) = (\triangle \alpha) \wedge \ast \Psi.
\end{equation}

\begin{example} \label{ex:Kaehler}
Let $(M,J,g)$ be a K\"ahler manifold with K\"ahler form $\om$. Then $\om$ is parallel, and the map $L: \Om^\ast(M) \to \Om^{\ast+2}(M)$, $\alpha \mapsto \alpha \wedge \om$ is called the {\em Lefschetz map}. In this case, $\Ll_\om = [L, d^\ast] = d^c = i (\bar{\p} - \p)$ is the {\em complex differential}, where $d = \p + \bar{\p}$ is the decomposition of the exterior differential into its holomorphic and anti-holomorphic part \cite[(4.1),(4.2)]{FN1956b}. 
Likewise, a straightforward calculation using (\ref{eq:kms}) shows that the adjoint map $ad_\om: \Om^\ast(M, TM) \to \Om^\ast(M, TM)$ coincides with the {\em Dolbeault differential} $\bar \p: \Om^{p,q}(M, TM) \to \Om^{p, q+1}(M, TM)$ \cite[Lemma 4]{FN1956b}. 
In particular, the cohomology algebras $H^\ast_\om(M)$ and $H^\ast_\om(M, TM)$ are finite dimensional if $M$ is closed.
\end{example}

Recall that for a closed oriented Riemannian manifold $(M,g)$ there is the {\em Hodge decomposition }of differential forms
\begin{equation} \label{Hodge-dec}
\Om^l(M) = \Hh^l(M) \oplus d\Om^{l-1} (M) \oplus d^\ast\Om^{l+1} (M),
\end{equation}
where $\Hh^l(M) \subset \Om^l(M)$ denotes the space of harmonic forms. Moreover $\dim \Hh^l(M) = \dim H^l_{dR}(M) = b^l(M)$ is a topological invariant.

Let us now compute the formal adjoint of $\Ll_\Psi$.

\begin{proposition}\label{prop:formal-adjointLK}
Let $(M, g)$, $\Psi \in \Om^{2k}(M)$ and $\hat \Psi := \p_g \Psi \in \Om^{2k-1}(M,TM)$ be as before. Then the formal adjoint
$\Ll_{\Psi;l}^\ast: \Om^l (M) \rightarrow \Om^{l-2k+1}(M)$ 
of 
$\Ll_{\Psi;l}: \Om^{l-2k+1}(M) \rightarrow \Om^l (M)$ 
is given by 
\begin{equation}\label{eq:formal-adjointLK}
\Ll_{\Psi;l}^\ast \alpha^l = (-1)^{n(n-l) + 1} \ast \Ll_{\Psi} \ast \alpha^l
\end{equation}
\end{proposition}

Observe the analogy of the formula for $\Ll_{\Psi;l}^\ast$ in (\ref{eq:formal-adjointLK}) and that for $d^\ast$ in (\ref{d*-formula}).

\begin{proof}
For 
$\alpha^l \in \Om^l(M)$ and 
$\beta^{l-2k+1} \in \Om^{l-2k+1}(M)$, we have 
\begin{align}
\nonumber
\la \Ll_\Psi \beta^{l-2k+1}, \alpha^l \ra_{L^2}
&
\stackrel{(\ref{eq:LK-d*})} =
\la (d^\ast \beta^{l-2k+1}) \wedge \Psi - d^\ast(\beta^{l-2k+1} \wedge \Psi), \alpha^l \ra_{L^2}\\
\label{eq:LKad-terms} &=
\int_M d^\ast \beta^{l-2k+1} \wedge \Psi \wedge \ast \alpha^l - \int_M \beta^{l-2k+1} \wedge \Psi \wedge \ast d \alpha^l.
\end{align}
Evaluating the two summands, we get by (\ref{prop-Hodge}) and (\ref{d*-formula}) 
\begin{align*}
\int_M &d^\ast \beta^{l-2k+1} \wedge \Psi \wedge \ast \alpha^l\\
&\stackrel{(\ref{prop-Hodge})}=
(-1)^{(n-l)l} \la d^\ast \beta^{l-2k+1}, \ast(\Psi \wedge \ast \alpha^l)\ra\\
&=
(-1)^{(n-l)l} \la \beta^{l-2k+1}, d \ast ( \Psi \wedge \ast \alpha^l) \ra_{L^2}\\
&\stackrel{(\ref{prop-Hodge}), (\ref{d*-formula})}=
(-1)^{n(n-l)} \la \beta^{l-2k+1},\ast d^\ast ( \Psi \wedge \ast \alpha^l) \ra_{L^2}.
\end{align*}
and
\begin{align*}
\int_M \beta^{l-2k+1} \wedge& \Psi \wedge \ast d \alpha^l\\
\stackrel{(\ref{prop-Hodge}), (\ref{d*-formula})}=&
(-1)^{l(n-l) + nl+1} 
\int_M \beta^{l-2k+1} \wedge \Psi \wedge d^\ast \ast \alpha^l \\
\stackrel{(\ref{prop-Hodge})}=&
(-1)^{n (n-l)} \left \la \beta^{l-2k+1}, \ast \left( \Psi \wedge d^\ast \ast \alpha^l \right) \right \ra_{L^2},
\end{align*}
Substituting these into (\ref{eq:LKad-terms}) yields
\begin{align*}
\la \Ll_\Psi \beta^{l-2k+1}, \alpha^l \ra_{L^2} =& (-1)^{n (n-l)} \left \la \beta^{l-2k+1}, \ast \left( d^\ast ( \Psi \wedge \ast \alpha^l) - \Psi \wedge d^\ast \ast \alpha^l  \right) \right \ra_{L^2}\\
=& (-1)^{n (n-l)} \left \la \beta^{l-2k+1}, -\ast \Ll_\Psi \ast \alpha^l \right \ra_{L^2},
\end{align*}
from which (\ref{eq:formal-adjointLK}) follows.
\end{proof}

We define the {\em space of $\Ll_\Psi$-harmonic forms} as
\begin{equation} \label{eq:LK-harmonic}
\begin{array}{lll} \Hh_\Psi^l(M) & := & \{ \alpha \in \Om^l(M) \mid \Ll_\Psi \alpha = \Ll_\Psi^\ast \alpha = 0\}\\[3mm]
& \stackrel{(\ref{eq:formal-adjointLK})}= & \{ \alpha \in \Om^l(M) \mid \Ll_\Psi \alpha = \Ll_\Psi \ast \alpha = 0\}
\end{array}
\end{equation}
Evidently, the Hodge-$\ast$ yields an isomorphism
\begin{equation} \label{eq:ast-Hodge}
\ast: \Hh_\Psi^l(M) \longrightarrow \Hh_\Psi^{n-l}(M).
\end{equation}

Since $\Hh^l_\Psi (M) \subset \ker \Ll_{\Psi;l+2k-1}$ and $\Hh^l_\Psi (M) \cap \Im(\Ll_{\Psi;l}) = 0$, there is a canonical injection
\begin{equation} \label{eq:inject-harmonic}
\imath_l: \Hh^l_\Psi(M) \hookrightarrow H^l_\Psi(M).
\end{equation}

This is analogous to the inclusion of harmonic forms into the deRham cohomology of a manifold, which for a closed manifold is an isomorphism due to the Hodge decomposition (\ref{Hodge-dec}). Therefore, one may hope that the maps $\imath_l$ are isomorphism as well. It is not clear if this is always true, but we shall give conditions which assure this to be the case and show that in the applications we have in mind, this condition is satisfied.

\begin{definition} \label{def:reg-cohom}
Let $(M, g)$ be an oriented Riemannian manifold, $\Psi \in \Om^{2k}(M)$ a parallel form and $\hat \Psi := \p_g \Psi \in {\mathcal {MC}}^{2k-1}(\Om^\ast(M, TM))$ as before.

We say that the differential $\Ll_\Psi$ is {\em $l$-regular }for $l \in \N$ if there is a direct sum decomposition
\begin{equation} \label{eq:condition-split}
\Om^l(M) = \ker (\Ll_{\Psi;l}^\ast) \oplus \Im(\Ll_{\Psi;l}).
\end{equation}
Furthermore, we call $\Ll_\Psi$ {\em regular} if it is $l$-regular for all $l \in \N$.
\end{definition}

A standard result from ellipticity theory states that $\Ll_\Psi$ is $l$-regular if the differential operator $\Ll_{\Psi;l}: \Om^{l-2k+1}(M) \to \Om^l(M)$ is elliptic, overdetermined elliptic or underdetermined elliptic, see e.g. \cite[p.464, 32 Corollary]{Besse1987}.

The following theorem now relates the cohomology $H^\ast_\Psi(M)$ to the $\Ll_\Psi$-harmonic forms $\Hh^\ast_\Psi(M)$.

\begin{theorem} \label{thm:cohom-harmonic}
Let $(M, g)$ be an oriented Riemannian manifold, $\Psi \in \Om^{2k}(M)$ a parallel form and $\hat \Psi := \p_g \Psi \in {\mathcal {MC}}^{2k-1}(\Om^\ast(M, TM))$ as before.
\begin{enumerate}
\item If $\Ll_\Psi$ is $l$-regular, then the map $\imath_l$ from (\ref{eq:inject-harmonic}) is an isomorphism.\item There are direct sum decompositions
\begin{eqnarray}
\label{eq:HL-decomp}
H_\Psi^l(M) & \cong & \Hh^l(M) \oplus H_\Psi^l(M)_d \oplus H_\Psi^l(M)_{d^\ast}\\
\label{eq:decomp-Hh_K}
\Hh_\Psi^l(M) & = & \Hh^l(M) \oplus \Hh_\Psi^l(M)_d \oplus \Hh_\Psi^l(M)_{d^\ast},
\end{eqnarray}
where $\Hh^l(M)$ is the space of harmonic $l$-forms on $M$, $H_\Psi^l(M)_d$ and $H_\Psi^l(M)_{d^\ast}$ are the cohomologies of $(d\Om^\ast(M), \Ll_\Psi)$ and $(d^\ast\Om^\ast(M), \Ll_\Psi)$, respectively, and where $\Hh_\Psi^l(M)_d := \Hh_\Psi^l(M) \cap d\Om^{l-1}(M)$, $\Hh_\Psi^l(M)_{d^\ast} := \Hh_\Psi^l(M) \cap d^\ast\Om^{l+1}(M)$. Moreover, the injective map $\imath_l$ from (\ref{eq:inject-harmonic}) preserves this decomposition, i.e.,
\[
\imath_l: \Hh_\Psi^l(M)_d \hookrightarrow H_\Psi^l(M)_d \quad \mbox{and} \quad \imath_l: \Hh_\Psi^l(M)_{d^\ast} \hookrightarrow H_\Psi^l(M)_{d^\ast}.
\]
\item
There are isomorphisms
\begin{align}
\label{eq:Hl-split1} d: H_\Psi^l(M)_{d^\ast} \to H_\Psi^{l+1}(M)_d \quad \mbox{and} \quad d^\ast: H_\Psi^l(M)_d \to H_\Psi^{l-1}(M)_{d^\ast}\\
\label{eq:Hl-split} d: \Hh_\Psi^l(M)_{d^\ast} \to \Hh_\Psi^{l+1}(M)_d \quad \mbox{and} \quad d^\ast: \Hh_\Psi^l(M)_d \to \Hh_\Psi^{l-1}(M)_{d^\ast}
\end{align}
\item If $\Ll_\Psi$ is $(l+1)$-regular and $(l-1)$-regular, then it is also $l$-regular.
\end{enumerate}
\end{theorem}

\begin{proof}
Note that (\ref{eq:condition-split}) and $\Ll_\Psi^2 = 0$ imply that $\ker \Ll_\Psi|_{\Om^l(M)} = \Hh_\Psi^l(M) \oplus \Ll_\Psi(\Om^{l-2k+1}(M))$, and from this, the first assertion follows.

If $\alpha_h \in \Hh^l(M)$ is harmonic, then $d^\ast \alpha_h = 0$, and $\alpha_h \wedge \Psi \in \Hh^{l+2k-1}(M)$ by (\ref{eq:Laplace-commute}), so that $d^\ast(\alpha_h \wedge \Psi) = 0$. This implies that $\Ll_\Psi(\alpha_h) = 0$, and since $\ast \alpha_h$ is also harmonic, it follows that $\Ll_\Psi(\ast \alpha_h) = 0$, whence $\Hh^l(M) \subset \Hh_\Psi^l(M)$.

Furthermore, $\Ll_\Psi$ anti-commutes with $d$ and $d^\ast$ by (\ref{eq:LK-commutes}), whence it follows that $\Ll_\Psi$ preserves the Hodge decomposition (\ref{Hodge-dec}), so that, in particular, the image of $\Ll_\Psi$ is contained in $d\Om^\ast(M) \oplus d^\ast \Om^\ast(M)$. From this, the decompositions (\ref{eq:HL-decomp}) and (\ref{eq:decomp-Hh_K}) and hence the second assertion follow.

For the third statement, we first show that
\begin{align}
\nonumber \ker \Ll_\Psi \cap d\Om^{l-1}(M) &= d (\ker \Ll_\Psi \cap d^\ast\Om^l(M)),\\
\label{eq:kerL-d-commute} \ker \Ll_\Psi \cap d^\ast\Om^{l+1}(M) &= d^\ast (\ker \Ll_\Psi \cap d\Om^l(M))\\
\nonumber \Im \Ll_\Psi \cap d\Om^{l-1}(M) &= d(\Im \Ll_\Psi \cap d^\ast \Om^l(M))\\
\nonumber \Im \Ll_\Psi \cap d^\ast\Om^{l-1}(M) &= d^\ast(\Im \Ll_\Psi \cap d \Om^l(M)).
\end{align}
Namely, if $d \alpha^{l-1} \in \ker (\Ll_\Psi)$ then by the Hodge decomposition we may assume w.l.o.g. that $\alpha^{l-1} = d^\ast \alpha^l$ with $\alpha^l \in \Om^l(M)$. Then
\[
0 = \Ll_\Psi(d\alpha^{l-1}) = \Ll_\Psi(dd^\ast \alpha^l) \stackrel{(\ref{eq:LK-commutes})}= dd^\ast \Ll_\Psi(\alpha^l),
\]
so that $0 = d^\ast \Ll_\Psi(\alpha^l) = - \Ll_\Psi(d^\ast \alpha^l) = - \Ll_\Psi(\alpha^{l-1})$. This shows that $\ker \Ll_\Psi \cap d\Om^{l-1}(M) \subset d (\ker \Ll_\Psi \cap d^\ast\Om^l(M))$, and the reverse inclusion follows from (\ref{eq:LK-commutes}). The remaining equalities in (\ref{eq:kerL-d-commute}) are shown analogously.

It follows that $d$ and $d^\ast$ induce the asserted maps in (\ref{eq:Hl-split1}). Observe that the compositions
\[
dd^\ast: H_\Psi^l(M)_d \to H_\Psi^l(M)_d \qquad \mbox{and} \qquad d^\ast d: H_\Psi^l(M)_{d^\ast} \to H_\Psi^l(M)_{d^\ast}
\]
are induced by the restriction of the Laplacian to $\ker \Ll_\Psi \cap d\Om^{l-1}(M)$ and $\ker \Ll_\Psi \cap d^\ast\Om^{l+1}(M)$, respectively, and hence they are isomorphisms. Thus, it follows that the maps in (\ref{eq:Hl-split1}) are isomorphisms as well.

In addition to the equations in (\ref{eq:kerL-d-commute}) it follows from (\ref{eq:formal-adjointLK}) that
\begin{align} 
\label{eq:kerL-d*-commute} \ker \Ll_\Psi^\ast|_{d\Om^{l-1}(M)} &= d (\ker \Ll_\Psi^\ast|_{d^\ast\Om^l(M)}), \qquad \mbox{and}\\
\nonumber \ker \Ll_\Psi^\ast|_{d^\ast\Om^{l+1}(M)} &= d^\ast (\ker \Ll_\Psi^\ast|_{d\Om^l(M)}).
\end{align}
Thus,
\begin{align*}
\Hh_\Psi^l(M)_d &= \ker \Ll_\Psi|_{d\Om^{l-1}(M)} \cap \ker \Ll_\Psi^\ast|_{d\Om^{l-1}(M)}\\
&\stackrel{(\ref{eq:kerL-d-commute}), (\ref{eq:kerL-d*-commute})} = d (\ker \Ll_\Psi|_{d^\ast\Om^l(M)}) \cap d (\ker \Ll_\Psi^\ast|_{d^\ast\Om^l(M)})\\
&\stackrel{(\ast)}= d \big(\ker \Ll_\Psi|_{d^\ast\Om^l(M)} \cap \ker \Ll_\Psi^\ast|_{d^\ast\Om^l(M)}\big) =d \Hh_\Psi^{l-1}(M)_{d^\ast},
\end{align*}
where at $(\ast)$ we used that $d: d^\ast \Om^l(M) \to \Om^{l-1}(M)$ is injective. Thus, the map $d: \Hh_\Psi^{l-1}(M)_{d^\ast} \to \Hh_\Psi^l(M)_d$ is an isomorphism, and the same holds for the other map in (\ref{eq:Hl-split}), and both are continuous in the Frech\'et topology.

For the fourth statement, suppose now that $\Ll_\Psi^\ast$ is $(l \pm 1)$-regular. Then by (\ref{eq:LK-commutes}) and (\ref{eq:kerL-d-commute})
\begin{align*}
d\Om^{l-1}(M) &= d\ker \Ll_{\Psi;l}^\ast|_{\Om^{l-1}(M)} \oplus d\Ll_{\Psi;l}(\Om^{l-2k}(M))\\
&= \ker \Ll_{\Psi;l}^\ast|_{d\Om^{l-1}(M)} \oplus \Ll_{\Psi;l}(d\Om^{l-2k}(M))\\
d^\ast\Om^{l+1}(M) &= d^\ast \ker \Ll_{\Psi;l}^\ast|_{\Om^{l+1}(M)} \oplus d^\ast \Ll_{\Psi;l}(\Om^{l+2-2k}(M))\\
&= \ker \Ll_{\Psi;l}^\ast|_{d^\ast\Om^{l+1}(M)} \oplus \Ll_{\Psi;l}(d^\ast\Om^{l+2-2k}(M)),
\end{align*}
and since $\Hh^l(M) \subset \ker \Ll_\Psi^\ast$ by the proof of (2), the Hodge decomposition of $\Om^l(M)$ implies
\begin{align*}
\Om^l(M) &= \Hh^l(M) \oplus d\Om^{l-1}(M) \oplus d^\ast\Om^{l+1}(M)\\
&= \Hh^l(M) \oplus \ker \Ll_{\Psi;l}^\ast|_{d\Om^{l-1}(M)} \oplus \Ll_{\Psi;l}(d\Om^{l-2k}(M))\\
&\qquad \oplus \ker \Ll_{\Psi;l}^\ast|_{d^\ast\Om^{l+1}(M)} \oplus \Ll_{\Psi;l}(d^\ast\Om^{l+2-2k}(M))\\
& = \ker \Ll_{\Psi;l}^\ast \oplus \Im(\Ll_{\Psi;l}),
\end{align*}
where the last equality follows since $\Ll_{\Psi;l}(\Hh^{l-2k+1}(M)) = 0$ and $\Ll_{\Psi;l}^\ast(\Hh^l(M)) = 0$. This shows that  $\Ll_\Psi^\ast$ is $l$-regular.
\end{proof}

In order to give an application of this result, recall that a {\em differential graded algebra (DGA)} is a graded associative algebra $V = (\bigoplus_{k \in \Z} V_k, \cdot)$ with a {\em differential}, i.e. a graded derivation $d$ of degree $1$ such that $d^2 = 0$. Since $d \in {\mathcal D}(V)$ is a Maurer-Cartan element, it defines the cohomology algebra (\ref{eq:cohom-GA}) $H^\ast(V, d) = \ker d/\Im d$.

A {\em DGA-homomorphism }between two DGAs $(V, d_V)$ and $(W, d_W)$ is a graded algebra morphism $\phi: (V, d_V) \to (W, d_W)$ which commutes with the differentials, i.e., $d_W \phi = \phi d_V$. In this case, $\phi$ induces an algebra homomorphism $\phi_\ast: H^\ast(V, d_V) \to H^\ast(W, d_W)$, and we call $\phi$ a {\em quasi-isomorphism} if $\phi_\ast$ is an isomorphism.

With this terminology, we may now formulate the following result.

\begin{corollary} \label{cor:formal}
Let $\Psi \in \Om^{2k}(M)$ be a parallel form on a closed Riemannian manifold $(M, g)$. Then the maps
\begin{equation} \label{eq:quasiiso}
(\Om^\ast(M), d) \longleftarrow (\ker \Ll_\Psi, d) \longrightarrow (H^\ast_\Psi(M), d)
\end{equation}
defined by the inclusion and the canonical projection, respectively, are quasi-isomorphisms.
\end{corollary}

Following standard terminology, the existence of the quasi-isomorphisms in (\ref{eq:quasiiso}) means that that $(H^\ast_\Psi(M), d)$ is {\em formally equivalent }to the deRham algebra $(\Om^\ast(M), d)$, and by \cite{Sullivan} this implies that $(H^\ast_\Psi(M), d)$ determines the rational homotopy type of $M$.

\begin{proof}
As $\Ll_\Psi$ is a derivation, it easily follows that $\ker \Ll_\Psi \subset \Om^\ast(M)$ is a subalgebra and $\Im \Ll_\Psi \subset \ker \Ll_\Psi$ is an ideal; as $\Ll_\Psi$ anti-commutes with $d$ by (\ref{eq:LK-commutes}), the maps in (\ref{eq:quasiiso}) are well defined DGA-homomorphisms.

Let $\alpha^l \in \ker d|_{\ker \Ll_\Psi}$, so that $\alpha^l = \alpha^l_h + d\beta^{l-1}$ with $\alpha^l_h \in \Hh^l(M)$ and $\beta^{l-1} \in \Om^{l-1}(M)$. 
Since $\alpha^l_h \in \ker \Ll_\Psi$, it follows that
\[
d\beta^{l-1} \in \ker \Ll_\Psi \cap d\Om^{l-1}(M) \stackrel{(\ref{eq:kerL-d-commute})}\subset d\ker \Ll_\Psi.
\]
That is, $\ker d|_{\ker \Ll_\Psi} = \Hh^\ast(M) \oplus d\ker \Ll_\Psi$, whence the inclusion 
$(\ker \Ll_\Psi, d) \hookrightarrow (\Om^\ast(M), d)$ is a quasi-isomorphism.

Since $0 \rightarrow (\Im \Ll_\Psi, d) \rightarrow (\ker \Ll_\Psi, d) \rightarrow (H^*_\Psi(M), d) \rightarrow 0$ is a short exact sequence of DGAs, it follows that the projection $\ker \Ll_\Psi \to H^\ast_\Psi(M)$ is a quasi-isomorphism iff $(\Im \Ll_\Psi, d|_{\Im \Ll_\Psi})$ has trivial cohomology. 
The Hodge decomposition of an element $\alpha^l = \alpha^l_h + d\beta^{l-1}�+ d^\ast \gamma^{l+1}$ yields by (\ref{eq:LK-commutes})
\[
\Ll_\Psi \alpha^l = - d \Ll_\Psi \beta^{l-1} - d^\ast \Ll_\Psi \gamma^{l+1},
\]
so that $\Ll_\Psi \alpha^l$ is closed iff it is contained in $d \Im \Ll_\Psi$, showing the triviality of the cohomology of $(\Im \Ll_\Psi, d)$ and hence the claim.
\end{proof}

We call a form $\Psi \in \Om^k(M)$ {\em multi-symplectic}, if $d\Psi = 0$ and if for all $v \in TM$
\begin{equation} \label{eq:Psi-nondeg}
\imath_v \Psi = 0 \Longleftrightarrow v = 0.
\end{equation}

Since for a parallel form $\Psi$ the distribution $\{ v \in TM \mid \imath_v \Psi = 0\}$ is parallel as well, it follows that for an irreducible Riemannian manifold any {\em parallel }form $\Psi \not \equiv 0$ is multi-symplectic.

\begin{lemma} \label{lem:elliptic multi-symplectic}
If $\Psi \in \Om^{2k}(M)$ is multi-symplectic, then the differential operator $\Ll_{\Psi;l}: \Om^{l-2k+1}(M) \to \Om^l(M)$ is overdetermined elliptic for $l = 2k-1$ and underdetermined elliptic for $l = n$.
\end{lemma}

\begin{proof}
By (\ref{eq:LK-nabla}), the symbol $\sigma_\xi (\Ll_{\Psi, l})$ of $\Ll_\Psi: \Om^{l-2k+1}(M) \rightarrow \Om^l (M)$ 
for $\xi \in T_p^\ast M$ is given by 
\begin{align} \label{eq:symbol of L_K}
\sigma_\xi (\Ll_{\Psi, l}): \Lambda^{l-2k+1} T_p^\ast M &\longrightarrow  \Lambda^l T_p^\ast M,\\
\nonumber
\alpha^{l-2k+1} &\longmapsto (\imath_{\xi^\#} \Psi) \wedge \alpha^{l+2k-1}.
\end{align}

If $l = 2k-1$ or $l = n$ then either the domain or the range of $\sigma_\xi (\Ll_{\Psi, l})$ is one dimensional, whence the injectivity or surjectivity of $\sigma_\xi (\Ll_{\Psi, l})$ is given unless $\sigma_\xi (\Ll_{\Psi, l}) = 0$.

For $l = 2k-1$, observe that $\sigma_\xi (\Ll_{\Psi, l}) (1) = (\imath_{\xi^\#} \Psi) \neq 0$ for $\xi \neq 0$ as $\Psi$ is multi-symplectic; for $l = n$,
\[
\sigma_\xi (\Ll_{\Psi, n}) (\ast \imath_{\xi^\#} \Psi) = (\imath_{\xi^\#} \Psi) \wedge (\ast \imath_{\xi^\#} \Psi) = |\imath_{\xi^\#} \Psi|^2 \vol_p,
\]
which again is non-zero for $\xi \neq 0$ as $\Psi$ is multi-symplectic.
\end{proof}

As an immediate consequence of this and Theorem \ref{thm:cohom-harmonic}, we obtain the 
\begin{corollary} \label{cor:elliptic multi-symplectic}
Let $\Psi \in \Om^{2k}(M)$ be a parallel multi-symplectic form on the Riemannian manifold $(M, g)$. Then
\[
H_\Psi^{2k-1}(\Om^\ast(M)) \cong \Hh^{2k-1}_\Psi(M).
\]
\end{corollary}

We also can make some statement for $H_\Psi^l(M)$ for special values of $l$.

\begin{proposition}\label{prop:chomom-LK_0,n}
Let $\Psi \in \Om^{2k}(M)$ be a parallel form on the oriented Riemannian manifold $(M, g)$. Then
\[
H_\Psi^0(\Om^\ast(M)) \cong \Hh_\Psi^0(M) = \{ f \in C^{\infty}(M) \mid \imath_{df^\#} \Psi = 0 \}.
\]

If $\Psi$ is multi-symplectic, then $\Hh_\Psi^0(M) = \Hh^0(M)$ and $H_\Psi^n (\Om^\ast(M)) \cong \Hh_\Psi^n(M) = \Hh^n(M)$.
\end{proposition}

\begin{proof}
Let $f \in \Om^0(M) = C^\infty(M)$. Then by (\ref{eq:LK-nabla}), we have 
\begin{align*}
\Ll_\Psi (f) = \imath_{df^\#} \Psi,
\end{align*}
which implies the statement for $H_\Psi^0(\Om^\ast(M))$.

If $\Psi$ is multi-symplectic, then $\imath_{df^\#} \Psi = 0$ iff $df = 0$, showing that $\Hh_\Psi^0(M) = \Hh^0(M)$. Moreover, using Lemma \ref{lem:elliptic multi-symplectic} for $l = n$, we see from Theorem \ref{thm:cohom-harmonic}(1) that $H_\Psi^n(\Om^\ast(M)) = \Hh_\Psi^n(M)$, and the latter space equals $\ast \Hh_\Psi^0(M)$ by (\ref{eq:ast-Hodge}), which by the above equals $\ast \Hh^0(M) = \Hh^n(M)$.
\end{proof}

Observe that $H_\Psi^0(\Om^\ast(M))$ is infinite dimensional if $\Psi$ is not multi-symplectic.

\begin{proposition} 
\label{prop:chomom-LK 1-form}
Let $\Psi \in \Om^{2k}(M)$ be a parallel multi-symplectic form on the oriented Riemannian manifold $(M, g)$. Then
\[
\ker (\Ll_{\Psi; 2k}) = \{ \alpha \in \Om^1(M) \mid \Ll_{\alpha^\#} (\ast \Psi) = 0 \quad \mbox{and} \quad d^\ast \alpha = 0\}.
\]
In particular, if $k \geq 2$ then $\ker \Ll_{\Psi; 2k} = \Hh^1_\Psi(M) \cong H^1_\Psi(M)$ and
\[
\Hh^{n-1}_\Psi(M) = \{ \alpha \in \Om^{n-1}(M) \mid \Ll_{(\ast\alpha)^\#} (\ast \Psi) = 0 \quad \mbox{and} \quad d\alpha = 0\}.
\]
\end{proposition}

\begin{proof} We calculate for $\alpha \in \Om^1(M)$ 
\begin{align*}
\Ll_\Psi(\alpha) \stackrel{(\ref{eq:LK-d*})}=& (d^\ast\alpha) \cdot \Psi - d^\ast(\alpha \wedge \Psi) = (d^\ast\alpha) \cdot \Psi + \ast d \ast (\alpha \wedge \Psi)\\
\stackrel{(\ref{eq:ast-contr})}=& (d^\ast\alpha) \cdot \Psi + \ast d (\imath_{\alpha^\#} \ast \Psi)\\
\stackrel{(\ref{prop-Hodge})}=& \ast \big((d^\ast\alpha) \cdot \ast \Psi + \Ll_{\alpha^\#} (\ast \Psi)\big),
\end{align*}
using again $d^\ast\Psi = 0$ and hence, $d (\imath_{\alpha^\#} \ast \Psi) = \Ll_{\alpha^\#} (\ast \Psi)$ in the last step.

Thus, $\Ll_\Psi(\alpha) = 0$ iff $\Ll_{\alpha^\#} (\ast \Psi) = - (d^\ast\alpha) \cdot \ast \Psi$. If this is the case, taking the exterior derivative implies that
\begin{align*}
- dd^\ast\alpha \wedge \ast \Psi \stackrel{d\ast \Psi = 0}=& - d((d^\ast\alpha) \cdot \ast \Psi)\\
 =& d\Ll_{\alpha^\#} (\ast \Psi) = \Ll_{\alpha^\#} (d\ast \Psi) = 0,
\end{align*}
whence
\[
0 = - dd^\ast\alpha \wedge \ast \Psi \stackrel{(\ref{eq:ast-contr})}= \ast \imath_{(dd^\ast\alpha)^\#} \Psi,
\]
so that $\imath_{(dd^\ast\alpha)^\#} \Psi = 0$ and hence, by (\ref{eq:Psi-nondeg}), $dd^\ast\alpha = 0$. This implies that $d^\ast \alpha = 0$ and hence, $\Ll_{\alpha^\#} (\ast \Psi) = 0$, showing the statement on $\ker \Ll_\Psi|_{\Om^1(M)}$.

The last statement then follows from the definition and (\ref{eq:ast-Hodge}).
\end{proof}

\section{The Fr\"olicher-Nijenhuis cohomology of $G_2$-manifolds}\label{sec:g2}

In this section we shall apply the cohomology definition from the preceding section to the parallel $4$-form in a $G_2$-manifold. We first collect some basic facts on the representation of the exceptional group $G_2$, see e.g. \cite{Humphreys}.

Let $V$ be an oriented $7$-dimensional vector space. A {\em $G_2$-structure on $V$ }is a form $\varphi \in \Lambda^3 V^\ast$ for which there is a positively oriented basis $(e_i)$ of $V$ with the dual basis $(e^i)$ of $V^\ast$ such that
\begin{equation} \label{varphi}
\varphi = e^{123} + e^{145} + e^{167} + e^{246} - e^{257} - e^{347} - e^{356},
\end{equation}
where $e^{i_1 \dots i_k}$ is short for $e^{i_1} \wedge \cdots \wedge e^{i_k}$. Setting $\vol := e^{1 \cdots 7}$, $\varphi$ uniquely determines an inner product $g_\varphi$ by the identity
\begin{equation} \label{eq:form-1def}
g_\varphi(u,v)\; \vol = \dfrac16 ((\imath_u \varphi) \wedge (\imath_v \varphi) \wedge \varphi),
\end{equation}
and it follows that any oriented basis $(e_i)$ for which (\ref{varphi}) holds is orthonormal w.r.t. $g_\varphi$. Thus, the Hodge-dual of $\varphi$ w.r.t. $g_\varphi$ is given by 
\begin{equation} \label{varphi*}
\ast \varphi = e^{4567} + e^{2367} + e^{2345} + e^{1357} - e^{1346} - e^{1256} - e^{1247}.
\end{equation}

The stabilizer of $\varphi$ is known to be the exceptional $14$-dimensional simple Lie group $G_2 \subset Gl(V)$, and the elements of $G_2$ preserve both $g_\varphi$ and $\vol$, i.e., $G_2 \subset SO(V, g_\varphi)$.

We summarize important known facts about the decomposition of exterior powers of $G_2$-modules into irreducible summands which are well known, see e.g. \cite[Section 2]{Kar2005}. We denote by $V_k$ the $k$-dimensional irreducible $G_2$-module if there is a unique such module. For instance, $V_7$ is the irreducible $7$-dimensional $G_2$-module from above, and $V_7^\ast \cong V_7$. For its exterior powers, we obtain the decompositions
\begin{equation} \label{eq:DiffForm-V7}
\begin{array}{rlrl}
\Lambda^0 V_7 \cong \Lambda^7 V_7 \cong V_1, \qquad
& \Lambda^2 V_7  \cong \Lambda^5 V_7 \cong V_7 \oplus  V_{14},\\[2mm]
\Lambda^1 V_7 \cong \Lambda^6 V_7 \cong V_7, \qquad
& \Lambda^3 V_7 \cong \Lambda^4 V_7 \cong V_1 \oplus V_7 \oplus V_{27},
\end{array}
\end{equation}
where $\Lambda^k V_7 \cong \Lambda^{7-k} V_7$ due to $G_2$-invariance of the Hodge isomorphism $\ast: \Lambda^k V_7 \to \Lambda^{7-k} V_7^\ast \cong \Lambda^{7-k} V_7$. We denote by $\Lambda^k_l V_7 \subset \Lambda^k V_7$ the subspace isomorphic to $V_l$ in the above notation. Evidently, $\Lambda^3_1 V_7$ and $\Lambda^4_1 V_7$ are spanned by $\varphi$ and $\ast \varphi$, respectively. For the remaining spaces in the decompositions of $\Lambda^k V_7$ we obtain the following descriptions.

\begin{align}
\nonumber
\Lambda^2_7 V_7 &= \{ \imath_v \varphi \mid v \in V_7\},\\
\nonumber
\Lambda^2_{14} V_7 &= \{ \alpha^2 \in \Lambda^2 V_7 \mid \alpha^2 \wedge \ast \varphi = 0\},\\
\nonumber
\Lambda^3_7 V_7 &= \{ \imath_v \ast \varphi \mid v \in V_7\} = \{ \ast (v^\flat \wedge \varphi) \mid v \in V_7\},\\
\label{decom-L-V7}
\Lambda^3_{27} V_7 &= \{ \alpha^3 \in \Lambda^3 V_7 \mid \alpha^3 \wedge \varphi = \alpha^3 \wedge \ast \varphi = 0\},\\
\nonumber
\Lambda^4_7 V_7 &= V_7 \wedge \varphi = \{ v \wedge \varphi \mid v \in V_7\},\\
\nonumber
\Lambda^4_{27} V_7 &= \{ \alpha^4 \in \Lambda^4 V_7 \mid \alpha^4 \wedge \varphi = \ast \alpha^4 \wedge \varphi = 0\},\\
\nonumber
\Lambda^5_7 V_7 &= V_7 \wedge \ast \varphi = \{ v \wedge \ast \varphi \mid v \in V_7\},\\
\nonumber
\Lambda^5_{14} V_7 &= \{ \alpha^5 \in \Lambda^5 V_7 \mid \alpha^5 \wedge (\imath_v \varphi) = 0 \; \mbox{for all $v \in V_7$}\}.
\end{align}

A {\em $G_2$-manifold }is a pair $(M, \varphi)$ consisting of a $7$-dimensional oriented manifold $M$ with a parallel $3$-form $\varphi$ such that at each $p \in M$ there is an oriented basis $(e_i)$ of $T_pM$ with dual basis $(e^i)$ of $T^\ast_pM$ such that $\varphi_p$ has the form (\ref{varphi}). Then $\varphi$ induces a Riemannian metric $g_\varphi$ on $M$, and (\ref{eq:DiffForm-V7}) induces a decomposition of differential forms and canonical projections by
\begin{equation} \label{eq:DiffForm-G2}
\Om^k_l(M) := \Gamma(M, \Lambda^k_l T^\ast M), \qquad \pi^k_l: \Om^k(M) \longrightarrow \Om^k_l(M).\end{equation}

By the Weitzenb\"ock formula, the Laplace operator $\triangle$ leaves sections of parallel subbundles of $\Lambda^k T^\ast M$ invariant, whence for $\alpha^k \in \Hh^k(M)$ each projection $\pi^k_l(\alpha^k)$ is also harmonic, so that
\begin{equation} \label{eq:bkl}
\begin{array}{lll}
\Hh^k(M) & = & \bigoplus_l (\Om^k_l (M) \cap \Hh^k(M)) =: \bigoplus_l \Hh^k_l(M),\\[5mm]
b^k(M) & = & \sum_l b^k_l(M),
\end{array}
\end{equation}
where $b^k_l(M) := \dim \Hh^k_l(M)$. Moreover, we let
\begin{equation} \label{eq:def-Om-kd}
\begin{array}{lll}
\Om^{k;d}_l(M) & := & \Om^k_l (M) \cap d\Om^{k-1}(M),\\[2mm]
\Om^{k;d^\ast}_l(M) & := & \Om^k_l (M) \cap d^\ast\Om^{k+1}(M)
\end{array}
\end{equation}
It is our aim to investigate these spaces for a $G_2$-manifold $(M, \varphi)$. First of all, the Hodge-$\ast$ gives isomorphisms
\begin{equation} \label{eq:Om-kd-*}
\Om^{k;d}_l(M) \xlongleftrightarrow{\ast} \Om^{7-k;d^\ast}_l(M)
\end{equation}
Since the decomposition of $\Lambda^k T^\ast M$ into its irreducible components is preserved by parallel translation, it follows from the Weitzenb\"ock formula (see e.g. \cite[(1.154)]{Besse1987}) that their sections are preserved by the Laplace operator, and since this operator also commutes with $d$ and $d^\ast$, we have
\begin{equation} \label{eq:Om-kd-Laplace}
\triangle \Om^{k;d}_l(M) = \Om^{k;d}_l(M) \qquad \mbox{and} \qquad \triangle \Om^{k;d^\ast}_l(M) = \Om^{k;d^\ast}_l(M)
\end{equation}
as well as
\begin{equation} \label{eq:Om-pi-Laplace}
\triangle \pi^k_l(\alpha^k) = \pi^k_l(\triangle \alpha^k) \qquad \mbox{for all $\alpha^k \in \Om^k(M)$}.
\end{equation}

\begin{definition} \label{def-V3new}
For $\beta^2 \in d\Om^1(M)$, we define the $4$-form
\begin{equation} \label{eq:def-V3-beta}
\alpha_{\beta^2}^4 := (d^\ast \beta^2) \wedge \varphi - d^\ast(\beta^2 \wedge \varphi).
\end{equation}
Moreover, define
\begin{equation} \label{eq:def-V3}
\begin{array}{l}
V_3^{d^\ast}(M) := \{ \ast \alpha_{\beta^2}^4 \mid \beta^2 \in d\Om^1(M)\} \quad \mbox{and}\\[2mm] V_4^d(M) := \{ \alpha_{\beta^2}^4 \mid \beta^2 \in d\Om^1(M)\}.
\end{array}
\end{equation}
\end{definition}

In the proofs of the following results we shall utilize some facts and identities provided in the Appendix \ref{sec:appendix}.

\begin{lemma} \label{lem:V3}
For every $\beta^2 \in d\Om^1(M)$, we have $\alpha_{\beta^2}^4 \in d\Om^3(M) \cap (\Om^{4;d}_{27}(M))^\perp$. Furthermore,
\begin{equation} \label{eq:V3-prop}
\alpha_{\beta^2}^4 \wedge \varphi = 0 \quad \mbox{and} \quad \ast \alpha_{\beta^2}^4 \wedge \varphi = -2 d\ast \beta^2 \in d\Om^5(M).
\end{equation}
\end{lemma}

\begin{proof}
First observe that
\[
d\alpha^4_{\beta^2} = (dd^\ast \beta^2) \wedge \varphi - dd^\ast (\beta^2 \wedge \varphi) \stackrel{d\beta^2 = 0}= (\triangle \beta^2) \wedge \varphi - \triangle (\beta^2 \wedge \varphi) \stackrel{(\ref{eq:Laplace-commute})} = 0.
\]
Note that
\[
\alpha^4_{\triangle \beta^2} = (\triangle d^\ast \beta^2) \wedge \varphi - d^\ast((\triangle \beta^2) \wedge \varphi) \stackrel{(\ref{eq:Laplace-commute})} = \triangle \alpha^4_{\beta^2},
\]
and since $d\Om^1(M) = \triangle d\Om^1(M)$, it follows that $\alpha^4_{\beta^2} \in \triangle \Om^4(M) = \Hh^4(M)^\perp$, and this together with $d\alpha^4_{\beta^2} = 0$ implies that $\alpha^4_{\beta^2} \in d\Om^3(M)$. If $\gamma^4 \in \Om^{4;d}_{27}(M)$, then
\[
\la \alpha^4_{\beta^2}, \gamma^4\ra_{L^2} = \int_M d^\ast \beta^2 \wedge \underbrace{\varphi \wedge \ast \gamma^4}_{=0} - \la \beta^2 \wedge \varphi, \underbrace{d \gamma^4}_{=0}\ra_{L^2} = 0,
\]
so that $\alpha^4_{\beta^2} \in (\Om^{4;d}_{27}(M))^\perp$. Furthermore,
\begin{align*}
\alpha^4_{\beta^2} \wedge \varphi &= (d^\ast \beta^2) \wedge \underbrace{\varphi \wedge \varphi}_{=0} + (\ast d \ast (\beta^2 \wedge \varphi)) \wedge \varphi = (d \ast (\beta^2 \wedge \varphi)) \wedge \ast \varphi\\
& = d(\ast (\beta^2 \wedge \varphi) \wedge \ast \varphi) \stackrel{(\ref{decom-L-V7})}= d(\pi^2_7(\ast (\beta^2 \wedge \varphi)) \wedge \ast \varphi)\\
& \stackrel{\text{Lemma}\; \ref{Lambda2-V7}}= 2 d(\pi^2_7(\beta^2) \wedge \ast \varphi) \stackrel{(\ref{decom-L-V7})}= 2 d(\beta^2 \wedge \ast \varphi) \stackrel{d\beta^2 = 0}= 0,
\end{align*}
and
\begin{align*}
\ast \alpha^4_{\beta^2} \wedge \varphi &= (\ast (d^\ast \beta^2 \wedge \varphi)) \wedge \varphi + (d \ast (\beta^2 \wedge \varphi)) \wedge \varphi\\
& \stackrel{(\ref{eq:form16}), \text{Lemma}\; \ref{Lambda2-V7}}= -4 \ast d^\ast \beta^2 + d ((2 \pi^2_7(\beta^2) - \pi^2_{14}(\beta^2)) \wedge \varphi)\\
& \stackrel{\text{Lemma}\; \ref{Lambda2-V7}}= -4 d \ast \beta^2 + d(4 \ast \pi^2_7(\beta^2) + \ast \pi^2_{14}(\beta^2))\\
& = -2 d \ast \beta^2 + d \ast((2 \pi^2_7(\beta^2) - \pi^2_{14}(\beta^2))\\
&\stackrel{\text{Lemma}\; \ref{Lambda2-V7}}= -2 d \ast \beta^2 + d \ast (\ast(\beta^2 \wedge \varphi)) = -2 d\ast \beta^2.
\end{align*}

\vspace{-7mm}
\end{proof}

\begin{proposition} \label{prop:diagrams}
For a closed $G_2$-manifold $(M, \varphi)$ and the spaces defined above, we have the following:
\begin{enumerate}
\item
$\Om^{2;d}_7(M) = \Om^{2;d}_{14}(M) = 0$ and $\Om^{5;d^\ast}_7(M) = \Om^{5;d^\ast}_{14}(M) = 0$.
\item
$\Om^{3;d}_1(M) = \Om^{3;d^\ast}_1(M) = 0$ and $\Om^{4;d}_1(M) = \Om^{4;d^\ast}_1(M) = 0$.
\item
$\Om^{3;d}_7(M) = \Om^{4;d^\ast}_7(M) = 0$.
\item
We have the following commutative diagrams of isomorphisms:

\begin{equation} \label{eq:diagrams}
\begin{xy}
\xymatrix{
\Om^{2;d^\ast}_7(M) \ar@/^2mm/[r]^d \ar@=[d] & dd^\ast(\Om^3_1(M)) \ar@/^2mm/[l]^{d^\ast} \ar@=[d] & 
\Om^{3;d^\ast}_7(M) \ar@/^2mm/[r]^d \ar@=[d] & dd^\ast \Om^4_1(M) \ar@/^2mm/[l]^{d^\ast} \ar@=[d]\\
\Om^{5;d}_7(M) \ar@/^2mm/[r]^{d^\ast} \ar@=[u]_\ast & d^\ast d(\Om^4_1(M)) \ar@/^2mm/[l]^d \ar@=[u]_\ast & \Om^{4;d}_7(M) \ar@/^2mm/[r]^{d^\ast} \ar@=[u]_\ast & d^\ast d \Om^3_1(M) \ar@/^2mm/[l]^d \ar@=[u]_\ast\\
\Om^{2;d^\ast}_{14}(M) \ar@/^2mm/[r]^d \ar@=[d] & \Om^{3;d}_{27}(M) \ar@/^2mm/[l]^{d^\ast} \ar@=[d] & 
\Om^{3,d^\ast}_{27}(M) \stackrel{\perp}\oplus V_3^{d^\ast}(M) \ar@/^4mm/[d]^d \ar@=[d]
\\
\Om^{5;d}_{14}(M) \ar@/^2mm/[r]^{d^\ast} \ar@=[u]_\ast & \Om^{4;d^\ast}_{27}(M) \ar@/^2mm/[l]^d \ar@=[u]_\ast & \Om^{4,d}_{27}(M) \stackrel{\perp}\oplus V_4^d(M) \ar@/^4mm/[u]^{d^\ast} \ar@=[u]_\ast
}
\end{xy}
\end{equation}
In fact, for the spaces in these diagrams we have the descriptions
\begin{equation} \label{def-U3-alt}
\Om^{2;d^\ast}_7(M) = \{ \imath_{(df)^\#} \varphi \mid f \in C^\infty(M)\} = d^\ast(\Om^3_1(M)),
\end{equation}
\begin{equation} \label{def-U4-alt}
\Om^{3;d^\ast}_7(M) = \{ \imath_{(df)^\#} \ast \varphi \mid f \in C^\infty(M)\} = d^\ast(\Om^4_1(M)),
\end{equation}
and
\begin{equation} \label{def-V3}
\Om^{3;d^\ast}_{27}(M) \stackrel{\perp}\oplus V_3^{d^\ast}(M) = \{ \alpha^3 \in d^\ast \Om^4(M) \mid \alpha^3 \wedge \ast \varphi = 0, d(\alpha^3 \wedge \varphi) = 0\}.
\end{equation}
\end{enumerate}
\end{proposition}

\begin{proof}
Let $\alpha^2 \in \Om^2_l(M)$, where $l = 7, 14$. Then according to Lemma \ref{Lambda2-V7} there are constants $c_l \in \R$ such that $\ast \alpha^2 = c_l (\alpha^2 \wedge \varphi)$. Thus, for $\alpha^2 \in \Om^2_l(M)$ we have
\begin{equation} \label{eq:d*alpha2}
d^\ast \alpha^2 \stackrel{(\ref{d*-formula})} = \ast d \ast \alpha^2 = c_l \ast d (\alpha^2 \wedge \varphi) = c_l \ast (d\alpha^2 \wedge \varphi).
\end{equation}
Therefore, if $\alpha^2 \in \Om^{2;d}_l(M)$, then (\ref{eq:d*alpha2}) implies that $d^\ast \alpha = 0$, whence $\alpha^2$ is both exact and coclosed and hence $\alpha^2 = 0$, showing that $\Om^{2;d}_l(M) = 0$ and therefore, $\Om^{5;d^\ast}_l(M) = 0$ by (\ref{eq:Om-kd-*}). This shows the first statement.

The second statement is immediate as $d(f \varphi) = df \wedge \varphi$ and $d(f \ast \varphi) = df \wedge \ast \varphi$ vanish iff $df = 0$ and hence $f \varphi$ and $f \ast \varphi$ are harmonic. Thus, $\Om^{k;d}_1(M) = 0$ for $k =3,4$, and the other two follow from (\ref{eq:Om-kd-*}).

For the third statement, let $\imath_{(\alpha^1)^\#}\ast \varphi \in \Om^{3;d}_7(M)$. Then $0 = d(\imath_{(\alpha^1)^\#}\ast \varphi) = \Ll_{(\alpha^1)^\#}\ast \varphi$, which implies that $(\alpha^1)^\#$ is a Killing vector field. As $G_2$-manifolds are Ricci-flat, 
Bochner's theorem implies that $(\alpha^1)^\#$ is parallel and hence, so is $\imath_{(\alpha^1)^\#}\ast \varphi$. Thus, $\imath_{(\alpha^1)^\#}\ast \varphi \in \Om^{3;d}_7(M)$ is harmonic, and since it is also exact, it must vanish. That is, $\Om^{3;d}_7(M) = 0$, and by (\ref{eq:Om-kd-*}) $\Om^{4;d^\ast}_7(M) = 0$ as well.

For the last part, we shall first prove (\ref{def-U3-alt}). Namely, for $f \in C^\infty(M)$
\[
\imath_{df^\#} \varphi \stackrel{(\ref{prop-Hodge}),(\ref{eq:ast-contr})} = \ast (df \wedge \ast \varphi) = \ast d(\ast f\varphi) \stackrel{(\ref{d*-formula})} = - d^\ast (f \varphi) \in d^\ast \Om^3_1(M).
\]
It follows that $\imath_{df^\#} \varphi \in d^\ast \Om^3_1(M) \subset \Om^{2;d^\ast}_7(M)$. Conversely, if $\alpha^2 = \imath_{(\alpha^1)^\#} \varphi \in \Om^{2;d^\ast}_7(M)$, then 
\[
0 = d^\ast \alpha^2 \stackrel{(\ref{d*-formula})} = \ast d \ast \imath_{(\alpha^1)^\#} \varphi \stackrel{(\ref{eq:ast-contr})} = \ast d(\alpha^1 \wedge \ast \varphi) = \ast (d\alpha^1 \wedge \ast \varphi),
\]
which implies that $d\alpha^1 \in \Om^{2;d}_{14}(M) = 0$, i.e., $\alpha^1 = \alpha^1_h + df$ for some $\alpha^1_h \in \Hh^1(M)$ and $f \in C^\infty(M)$, i.e., so that $\alpha^2 = \imath_{df^\#} \varphi + \imath_{(\alpha^1_h)^\#} \varphi$. Since we already showed that $\imath_{df^\#} \varphi \in \Om^{2;d^\ast}_7(M)$ and because $\imath_{(\alpha^1_h)^\#} \varphi$ is harmonic, it follows that $\alpha^1_h = 0$, and hence, $\alpha^2 \in \Om^{2;d^\ast}_7(M)$ iff $\alpha^2 = \imath_{df^\#} \varphi = - d^\ast (f \varphi)$ for some $f \in C^\infty(M)$ which shows (\ref{def-U3-alt}), so that $d\Om^{2;d^\ast}_7(M) = dd^\ast (\Om^3_1(M))$ and $d^\ast dd^\ast (\Om^3_1(M)) = \triangle \Om^{2;d^\ast}_7(M) = \Om^{2;d^\ast}_7(M)$.

The second diagram in (\ref{eq:diagrams}) and (\ref{def-U4-alt}) is dealt with in a similar fashion; in fact, by replacing $\alpha^2 = \imath_{(\alpha^1)^\#} \varphi$ by $\alpha^3 = \imath_{(\alpha^1)^\#} \ast \varphi$, we show by literally the same proof that (\ref{def-U4-alt}) holds, whence the second diagram commutes as well.

Now let us consider the third diagram in (\ref{eq:diagrams}). If $\alpha^2 \in \Om^{2;d^\ast}_{14}(M)$, then $d\alpha^2 \wedge \ast \varphi = 0$ and  (\ref{eq:d*alpha2}) implies that $d\alpha^2 \wedge \varphi = 0$, which shows that
\begin{equation} \label{eq:Om3d27subs}
d \Om^{2;d^\ast}_{14}(M) \subset \Om^{3;d}_{27}(M).
\end{equation}
Conversely, suppose that $\alpha^3 \in \Om^{3;d}_{27}(M)$, and let $\alpha^2 := d^\ast \alpha^3$. Since $\alpha^2 \in d^\ast \Om^3(M) \subset \triangle \Om^2(M)$, it follows from (\ref{eq:Laplace-commute}) that $\ast (\alpha^2 \wedge \varphi) \in \triangle \Om^2(M) = (\Hh^2(M))^\perp$. Moreover,
\begin{equation} \label{eq:d*alpha2-1}
d^\ast (\ast (\alpha^2 \wedge \varphi)) \stackrel{(\ref{d*-formula}), (\ref{prop-Hodge})} = \ast (d\alpha^2 \wedge \varphi) = \ast (\triangle \alpha^3 \wedge \varphi) \stackrel{(\ref{eq:Laplace-commute})}= \ast \triangle (\alpha^3 \wedge \varphi) = 0.
\end{equation}
Thus, $\ast (\alpha^2 \wedge \varphi) \in d^\ast \Om^3(M)$ and therefore,
\[
3 \pi^2_7(\alpha^2) \stackrel{\text{Lemma}\; \ref{Lambda2-V7}}= \alpha^2 + \ast (\alpha^2 \wedge \varphi) \in \Om^{2;d^\ast}_7(M),
\]
so that by (\ref{def-U3-alt}) $\pi^2_7(\alpha^2) = \imath_{(df)^\#} \varphi$ for some $f \in C^\infty(M)$. But then,
\[
\alpha^2 \wedge \ast \varphi = \pi^2_7(\alpha^2) \wedge \ast \varphi = \imath_{(df)^\#} \varphi  \wedge \ast \varphi \stackrel{(\ref{eq:form15})} = 3 \ast df,
\]
whence, 
\[
3 d \ast df = d\alpha^2 \wedge \ast \varphi = \triangle \alpha^3 \wedge \ast \varphi = 0
\]
as $\alpha^3 \in \Om^3_{27}(M)$, and hence $d^\ast df = 0$, so that $df = 0$ and hence, $\pi^2_7(\alpha^2) = 0$ which implies that $\alpha^2 \in \Om^{2;d^\ast}_{14}(M)$.

Thus, we have shown that $d^\ast \Om^{3;d}_{27}(M) \subset \Om^{2;d^\ast}_{14}(M)$, which together with (\ref{eq:Om3d27subs}) shows that the maps in the third diagram are isomorphisms.

We now show (\ref{def-V3}). The inclusion $\subset$ is obvious from the definition of $\Om^{3;d^\ast}_{27}(M)$ and Lemma \ref{lem:V3}. For the converse, let $\alpha^3 \in d^\ast \Om^4(M)$ be such that $\alpha^3 \wedge \ast \varphi = 0$ and $d\alpha^3 \wedge \varphi = 0$. Let $\tilde \alpha^3 \in d^\ast \Om^4(M)$ be such that $\alpha^3 = \triangle \tilde \alpha^3$, let $\beta^2 := d \ast (\tilde \alpha^3 \wedge \varphi) \in d\Om^1(M)$, and
\[
\widehat \alpha^3 := \alpha^3 + \frac12 \ast \alpha^4_{\beta^2}
\]
with $\alpha^4_{\beta^2}$ from (\ref{eq:def-V3-beta}). Since $\alpha^4_{\beta^2} \in d\Om^3(M)$ by Lemma \ref{lem:V3}, it follows that $\widehat \alpha^3 \in d^\ast \Om^4(M)$. Moreover,
\[
\widehat \alpha^3 \wedge \ast \varphi = \alpha^3 \wedge \ast \varphi + \frac12 \ast \alpha^4_{\beta^2} \wedge \ast \varphi = 0 + \frac12 \alpha^4_{\beta^2} \wedge \varphi \stackrel{(\ref{eq:V3-prop})}= 0,
\]
and
\begin{align*}
\widehat \alpha^3 \wedge \varphi &= \alpha^3 \wedge \varphi + \frac12 (\ast \alpha^4_{\beta^2}) \wedge \varphi \stackrel{(\ref{eq:V3-prop})}= \alpha^3 \wedge \varphi - d \ast \beta^2\\
& = \alpha^3 \wedge \varphi - dd^\ast (\tilde \alpha^3 \wedge \varphi) = \alpha^3 \wedge \varphi - \triangle (\tilde \alpha^3 \wedge \varphi) + d^\ast d(\tilde \alpha^3 \wedge \varphi)\\
& \stackrel{(\ref{eq:Laplace-commute})}= \alpha^3 \wedge \varphi - (\triangle \tilde \alpha^3) \wedge \varphi + d^\ast d(\tilde \alpha^3 \wedge \varphi) = d^\ast d(\tilde \alpha^3 \wedge \varphi).
\end{align*}
Now
\[
\triangle d(\tilde \alpha^3 \wedge \varphi) \stackrel{(\ref{eq:Laplace-commute})}=  d((\triangle \tilde \alpha^3) \wedge \varphi) = d(\alpha^3 \wedge \varphi) = 0,
\]
so that $d(\tilde \alpha^3 \wedge \varphi) \in \Hh^7(M) \cap d\Om^6(M) = 0$. Thus, $d(\tilde \alpha^3 \wedge \varphi) = 0$, whence $\widehat \alpha^3 \wedge \varphi = 0$.

All of this now implies that $\widehat \alpha^3 \in \Om^{3;d^\ast}_{27}(M)$ and therefore,
\[
\alpha^3 = \widehat \alpha^3 - \frac12 \ast \alpha^4_{\beta^2} \in \Om^{3;d^\ast}_{27}(M) \oplus V_3^{d^\ast}(M),
\]
which shows (\ref{def-V3}). Thus, in order to show that the maps in the last diagram in (\ref{eq:diagrams}) are isomorphisms, we have to show that $\ast d$ preserves $\Om^{3;d^\ast}_{27}(M) \oplus V_3^{d^\ast}(M)$. If $\alpha^3$ is an element of this space, then evidently, $\ast d \alpha^3 = d^\ast \ast \alpha^3 \in d^\ast \Om^4(M)$. Moreover,
\begin{align*}
\ast d\alpha^3 \wedge \ast \varphi &= d\alpha^3 \wedge \varphi = 0\\
(d\ast d\alpha^3) \wedge \varphi &= (\ast d \ast d\alpha^3) \wedge \ast \varphi \stackrel{d^\ast \alpha^3 = 0} = (\triangle \alpha^3) \wedge \ast \varphi \stackrel{(\ref{eq:Laplace-commute})}= \triangle (\alpha^3 \wedge \ast \varphi) = 0.
\end{align*}
Thus, $\ast d$ maps $\Om^{3;d^\ast}_{27}(M) \oplus V_3^{d^\ast}(M)$ to itself, and since the restriction of $(\ast d)^2$ to $\Om^{3;d^\ast}_{27}(M) \oplus V_3^{d^\ast}(M)$ coincides with the Laplacian and hence is an isomorphism, if follows that $\ast d: \Om^{3;d^\ast}_{27}(M) \oplus V_3^{d^\ast}(M) \to \Om^{3;d^\ast}_{27}(M) \oplus V_3^{d^\ast}(M)$ is an isomorphism as well.
\end{proof}

\begin{proposition} \label{prop:inf-dim}
For any $G_2$-manifold, all spaces in the diagrams (\ref{eq:diagrams}) are infinite dimensional.
\end{proposition}

\begin{proof}
By (\ref{def-U3-alt}), (\ref{def-U4-alt}) and (\ref{eq:V3-prop}), there are isomorphisms
\begin{align*}
dC^\infty(M) \xrightarrow{\ \cong\ } \Om^{2;d^\ast}_7(M) & \qquad df \longmapsto \imath_{(df)^\#} \varphi\\
dC^\infty(M) \xrightarrow{\ \cong\ } \Om^{3;d^\ast}_7(M) & \qquad df \longmapsto \imath_{(df)^\#} \ast \varphi\\
d\Om^1(M) \xrightarrow{\ \cong\ } V_3^{d^\ast}(M) & \qquad \beta^2 \longmapsto \ast \alpha^4_{\beta^2},
\end{align*}
showing that these spaces are infinite dimensional. In order to show that $\Om^{2;d^\ast}_{14}(M)$ and $\Om^{3;d^\ast}_{27}(M)$ are infinite dimensional, we assert that there are direct sum decompositions
\begin{align}
\label{eq:Om14-2} \Om^2_{14}(M) = \Hh^2_{14}(M) \oplus \Om^{2;d^\ast}_{14}(M) \oplus \pi^2_{14}(d\Om^1(M)),
\\
\label{eq:Om27-3} \Om^3_{27}(M) = \Hh^3_{27}(M) \oplus \Om^{3;d^\ast}_{27}(M) \oplus \pi^3_{27}(d\Om^2(M)).
\end{align}
If these decompositions hold, then we apply Lemma \ref{lem:codim-inf} to the differential operators $\Om^1(M) \ni \alpha^1 \mapsto \pi^2_{14}(d \alpha^1) \in \Om^2_{14}(M)$ and $\Om^2(M) \ni \alpha^2 \mapsto \pi^3_{27}(d\alpha^2) \in \Om^3_{27}(M)$, respectively, to conclude that $\Om^{2;d^\ast}_{14}(M)$ and $\Om^{3;d^\ast}_{27}(M)$ are infinite dimensional, as the space of harmonic forms is finite dimensional.

To see the decompositions (\ref{eq:Om14-2}) and (\ref{eq:Om27-3}), we consider the differential operators $\phi_1: \Om^1(M) \to \Om^2_{14}(M) \oplus C^\infty(M)$ and $\phi_2: \Om^2(M) \to \Om^3_{27}(M) \oplus \Om^1(M)$ given by
\[
\phi_1(\alpha^1) := \pi^2_{14}(d\alpha^1) + d^\ast \alpha^1, \qquad \phi_2(\alpha^2) := \pi^3_{27}(d\alpha^2) + d^\ast \alpha^2.
\]
Their formal adjoints are given by $\phi_1^\ast: \Om^2_{14}(M) \oplus C^\infty(M) \to \Om^1(M)$ and $\phi_2^\ast: \Om^3_{27}(M) \oplus \Om^1(M) \to \Om^2(M)$ as
\[
\phi_1^\ast(\beta^2_{14}, f) = d^\ast \beta^2_{14} + df \quad \mbox{and} \quad \phi_2^\ast(\beta^3_{27}, \beta^1) = d^\ast \beta^3_{27} + d\beta^1.
\]
By the Hodge decomposition $(\beta^2_{14}, f) \in \ker \phi_1^\ast$ iff $d^\ast \beta^2_{14} = 0$ and $df = 0$. In this case, since there is an orthogonal decompositions $\Hh^2(M) = \Hh^2_7(M) \oplus \Hh^2_{14}(M)$, it follows that the harmonic part of $\beta^2_{14}$ must be an element of $\Hh^2_{14}(M)$, whence the coexact part must be an element of $\Om^2_{14}(M)$ as well. Thus, it follows that
\[
\ker \phi_1^\ast = \big(\Hh^2_{14}(M) \oplus \Om^{2;d^\ast}_{14}(M)\big) \oplus \Hh^0(M),
\]
and analogously,
\[
\ker \phi_2^\ast = \big(\Hh^3_{27}(M) \oplus \Om^{3;d^\ast}_{27}(M)\big) \oplus \ker d|_{\Om^1(M)}.
\]
Thus, if we can show that $\phi_1, \phi_2$ are overdetermined elliptic, then $\Om^2_{14}(M) \oplus C^\infty(M) = \phi_1(\Om^1(M)) \oplus \ker \phi_1^\ast$ and $\Om^3_{27}(M) \oplus \Om^1(M) = \phi_2(\Om^2(M)) \oplus \ker \phi_2^\ast$, and from this, (\ref{eq:Om14-2}) and (\ref{eq:Om27-3}) follows. 

In order to show that $\phi_1$ and $\phi_2$ are overdetermined elliptic, we calculate their symbols for $0 \neq \xi \in T_p^\ast M$ as
\[
\sigma_\xi^1(\alpha^1) = \pi^2_{14}(\xi \wedge \alpha^1) - \imath_{\xi^\#} \alpha^1, \qquad \sigma_\xi^2(\alpha^2) = \pi^3_{27}(\xi \wedge \alpha^2) - \imath_{\xi^\#} \alpha^2.
\]

Let $\alpha^1 \in \ker \sigma_\xi^1$. Then $\xi \wedge \alpha^1 = \imath_v \varphi \in \Lambda^2_7 T_p^\ast M$ for some $v \in T_pM$. But then, (\ref{eq:form-1def}) implies that
\[
6 \|v\|^2\; \vol = (\xi \wedge \alpha^1)^2 \wedge \varphi = 0 \Rightarrow v = 0 \Rightarrow \xi \wedge \alpha^1 = 0,
\]
so that $\alpha^1 = c \xi$ for some $c \in \R$. But then, $0 = \imath_{\xi^\#} \alpha^1 = c \|\xi\|^2$ and $\xi \neq 0$ implies that $c = 0$ and hence, $\alpha^1 = 0$, so that $\ker \sigma_\xi^1 = 0$.

Similarly, $\alpha^2 \in \ker \sigma^2_\xi$ iff $\imath_{\xi^\#} \alpha^2 = 0$ and $\xi \wedge \alpha^2 \in \Lambda^3_7 T_p^\ast M \oplus \Lambda^3_1 T_p^\ast M$. The second equation implies that there is a $c \in \R$ and $v \in T_pM$ such that
\[
\xi \wedge \alpha^2 = c \varphi + \imath_v \ast \varphi \Rightarrow 0 = \xi \wedge (c \varphi + \imath_v \ast \varphi) \wedge \varphi \stackrel{(\ref{eq:form16})} = 4 \xi \wedge \ast v^\flat = 4g(\xi^\#,v)\; \vol,
\]
so that $\xi^\#, v$ are orthogonal. As $G_2$ acts transitively on orthonormal pairs, we may assume w.l.o.g. that $\xi = c_1 e_1$, $v = c_2 e_2$, whence
\begin{align*}
c_1 e^1 \wedge \alpha^2 & = c \varphi + c_2 \imath_{e_2} \ast \varphi\\
\Rightarrow 0 & = e^1 \wedge (c \varphi + c_2 \imath_{e_2} \ast \varphi)\\
& = c (e^{1246} - e^{1257} - e^{1347} - e^{1356}) + c_2 (e^{1367} + e^{1345}),
\end{align*}
and from this, $c = c_2 = 0$ and hence, $\xi \wedge \alpha^2 = 0$ follows, and this together with $\imath_{\xi^\#} \alpha^2 = 0$ implies that $\alpha^2 = 0$ and hence the injectivity of $\sigma^2_\xi$.
\end{proof}

We are now ready to calculate the Fr\"olicher-Nijenhuis cohomology algebra of a $G_2$-manifold.

\begin{theorem} \label{thm:cohomG2}
Let $(M, \varphi)$ be a closed $G_2$-manifold, so that $\ast \varphi \in \Om^4(M)$ is parallel. Then the differential $\Ll_{\ast \varphi}$ is regular, and the cohomology algebra $H^\ast_{\ast \varphi}(\Om^\ast(M))$ is given as follows:
\[
\begin{array}{llllllllll}
H_{\ast \varphi}^2(M) \cong & \Hh^2(M) & \oplus & \Om^{2;d^\ast}_7(M) & \oplus & \Om^{2;d^\ast}_{14}(M)\\[1mm]
H_{\ast \varphi}^3(M) \cong & \Hh^3(M) & \oplus & dd^\ast \Om^3_1(M) & \oplus & \Om^{3;d}_{27}(M) & \oplus & \Om^{3;d^\ast}_{27}(M) & \oplus & V^3_{d^\ast}(M)\\[1mm]
H_{\ast \varphi}^4(M) \cong & \Hh^4(M) & \oplus & d^\ast d \Om^4_1(M) & \oplus & \Om^{4;d^\ast}_{27}(M) & \oplus & \Om^{4;d}_{27}(M) & \oplus & V^4_d(M)\\[1mm]
H_{\ast \varphi}^5(M) \cong & \Hh^5(M) & \oplus & \Om^{5;d}_7(M) & \oplus & \Om^{5;d}_{14}(M)
\end{array}
\]
with the definitions in (\ref{eq:DiffForm-G2}) and (\ref{eq:def-V3}), and $H_{\ast \varphi}^k(M) \cong  \Hh^k(M)$ for $k = 0,1,6,7$. Moreover, the summands in this decomposition are $L^2$-orthogonal.
\end{theorem}

Observe that by Corollary \ref{prop:inf-dim}, all summands (apart from $\Hh^k(M)$) in the above decomposition are infinite dimensional.

\begin{proof}
We begin by showing the regularity of $\Ll_{\ast \varphi}$. For $l < 3$ and $l > 7$, this is obvious as then $\Ll_{\ast \varphi; l} = 0$. By Lemma \ref{lem:elliptic multi-symplectic}, $\Ll_{\ast \varphi,l}$ is overdetermined elliptic for $l = 3$ and underdetermined elliptic for $l=7$, whence $\Ll_{\ast \varphi}$ is also $3$-and $7$-regular. 

We assert that for $0 \neq \xi \in T_pM$ the symbol $\sigma_\xi(\Ll_{\ast \varphi,l}): \Lambda^{l-3}T_p^\ast M \to \Lambda^l T_p^\ast M$ is injective for $l = 4$ and surjective for $l=6$. Namely, by (\ref{eq:symbol of L_K}) we have
\[
\sigma_\xi(\Ll_{\ast \varphi,l})(\alpha^{l-3}) = (\imath_{\xi^\#} \ast \varphi) \wedge \alpha^{l-3}.
\]
Rescaling $\xi$ to a unit vector and using that $G_2$ acts transitively on the unit sphere, we may assume w.l.o.g. that $\xi = e^1$ for an orthonormal basis $(e_i)$ of $T_pM$ for which (\ref{varphi*}) holds. Now
\begin{align*}
\sigma_{e^1}(\Ll_{\ast \varphi,4}) (e^1) &= (\imath_{e^1} \ast \varphi) \wedge e^1 
= -(e^{1357} - e^{1346} - e^{1256} - e^{1247}) \neq 0\\
\sigma_{e^1}(\Ll_{\ast \varphi,4}) (e^2) &=  (\imath_{e^1} \ast \varphi) \wedge e^2 
= -(e^{2357}-e^{2346})\neq 0\\
\sigma_{e^1}(\Ll_{\ast \varphi,6}) (\varphi) &=  (\imath_{e^1} \ast \varphi) \wedge \varphi \stackrel{(\ref{eq:form16})}= 4 \ast e^1\\
\sigma_{e^1}(\Ll_{\ast \varphi,6}) (e^{146}) &=  (\imath_{e^1} \ast \varphi) \wedge e^{146} = e^{134567} = - \ast e^2
\end{align*}
The stabilizer of $e_1$ in $G_2$ is isomorphic to $SU(3) \subset G_2$, acting trivially on $\R e_1$ and via the standard $6$-dimensional representation on $(e_1)^\perp$. The image and the kernel of $\sigma_{e^1}(\Ll_{\ast \varphi,l})$ are invariant under this stabilizer. By the above, $\ker \sigma_{e^1}(\Ll_{\ast \varphi,4})$ contains neither $e^1$ nor $e^2$, so by the $SU(3)$-invariance of the kernel we must have $\ker \sigma_{e^1}(\Ll_{\ast \varphi,4}) = 0$. Likewise, the image of $\sigma_{e^1}(\Ll_{\ast \varphi,4})$ contains $\ast e^1$ and $\ast e^2$, so by the $SU(3)$-invariance it is all of $\ast T_p^\ast M = \Lambda^6 T_p^\ast M$, showing the above assertion.

Thus, $\Ll_{\ast \varphi,l}$ is overdetermined elliptic for $l = 4$ and underdetermined elliptic for $l=6$, whence $\Ll_{\ast \varphi}$ is $4$-regular and $6$-regular and hence, it is also $5$-regular by Theorem \ref{thm:cohom-harmonic}(4). Therefore, the regularity of $\Ll_{\ast \varphi,l}$ is established, whence by Theorem \ref{thm:cohom-harmonic}(1), $H^l_{\ast \varphi}(M) \cong \Hh^l_{\ast \varphi}(M)$.

For $l = 0,7$, $\Hh^l_{\ast \varphi}(M) \cong \Hh^l(M)$ by Proposition \ref{prop:chomom-LK_0,n}.

For $l = 1$, $H^1_{\ast \varphi}(M) = \ker \Ll_{\ast \varphi}|_{\Om^1(M)}$. Thus, by Proposition \ref{prop:chomom-LK 1-form}, $\alpha \in H^1_{\ast \varphi}(M)$ implies that $\Ll_{\alpha^\#}(\varphi) = 0$, which in turn implies that $\alpha^\#$ is a Killing vector field. Since a $G_2$-manifold is Ricci flat, it follows by Bochner's theorem that $\alpha^\#$ is parallel, whence so is $\alpha$. In particular, $\alpha$ is harmonic, showing that $H^1_{\ast \varphi}(M) \cong \Hh^1(M)$. For $l = 6$, we have $H^6_{\ast \varphi}(M) \cong \ast H^1_{\ast \varphi}(M) \cong \ast \Hh^1(M) \cong \Hh^6(M)$. This shows that $\Hh^l_{\ast \varphi}(M) \cong \Hh^l(M)$ for $l = 1, 6$.

Next, for $l = 2$, we have $\Hh^2_{\ast \varphi}(M)_d = 0$ by (\ref{eq:Hl-split1}). Thus, we need to determine
\[
\Hh^2_{\ast \varphi}(M)_{d^\ast} = \{ \alpha^2 \in d^\ast \Om^3(M) \mid d^\ast (\alpha^2 \wedge \ast \varphi) = 0\}.
\]
Observe that $\Om^{2;d^\ast}_7(M) \oplus \Om^{2;d^\ast}_{14}(M) \subset \Hh^2_{\ast \varphi}(M)_{d^\ast}$. Namely, for $\Om^{2;d^\ast}_{14}(M)$ this is obvious as $\alpha^2 \wedge \ast \varphi = 0$ for $\alpha^2 \in \Om^2_{14}(M)$. Moreover, by (\ref{def-U3-alt}), we have for $\Om^{2;d^\ast}_7(M) \ni \alpha^2 = \imath_{(df)^\#} \varphi$:
\[
d^\ast (\alpha^2 \wedge \ast \varphi) = d^\ast(\imath_{(df)^\#} \varphi \wedge \ast \varphi) \stackrel{(\ref{eq:form15})} = 3 d^\ast \ast df = 0.
\]

Conversely, let $\alpha^2 \in \Hh^2_{\ast \varphi}(M)_{d^\ast}$. Then there is an $f \in C^\infty(M)$ such that
\[
d(\alpha^2 \wedge \ast \varphi) = \triangle f \vol.
\]
Then by the above, $\imath_{(df)^\#} \varphi \in \Hh^2_{\ast \varphi}(M)_{d^\ast}$, so that
\[
\beta^2 := \alpha^2 + \dfrac13 \imath_{(df)^\#} \varphi \in \Hh^2_{\ast \varphi}(M)_{d^\ast},
\]
and hence,
\[
0 = \Ll_{\ast \varphi}(\beta^2) = d^\ast(\beta^2 \wedge \ast \varphi).
\]
On the other hand,
\begin{align*}
d(\beta^2 \wedge \ast \varphi) &= d(\alpha^2 \wedge \ast \varphi) + \dfrac13 d(\imath_{(df)^\#} \varphi \wedge \ast \varphi)\\
& \stackrel{(\ref{eq:form15})} = \triangle f \vol + d \ast df \stackrel{(\ref{prop-Hodge}), (\ref{d*-formula})} = \triangle f \vol - \triangle f \vol = 0.
\end{align*}

Thus, $\beta^2 \wedge \ast \varphi \in \Hh^6(M)$. On the other hand, $\beta^2 \in d^\ast \Om^3(M) \subset \triangle \Om^2(M)$, whence by (\ref{eq:Laplace-commute}), it follows that $\beta^2 \wedge \ast \varphi \in \triangle \Om^6(M) = (\Hh^6(M))^\perp$. Thus, $\beta^2 \wedge \ast \varphi = 0$, so that $\beta^2 \in \Om^{2;d^\ast}_{14}(M)$.

Therefore, $\alpha^2 = - \dfrac13 \imath_{(df)^\#} \varphi + \beta^2 \in \Om^{2;d^\ast}_7(M) \oplus \Om^{2;d^\ast}_{14}(M)$, showing that $\Hh^2_{\ast \varphi}(M)_{d^\ast}$ is of the asserted form.

For $l = 3$, it follows from (\ref{eq:Hl-split1}) that $\Hh^3_{\ast \varphi}(M)_d = d \Hh^2_{\ast \varphi}(M)_{d^\ast}$ which by Proposition \ref{prop:diagrams} equals $dd^\ast \Om^3_1(M) \oplus \Om^{3;d}_{27}(M)$, whence we need to calculate $\Hh^3_{\ast \varphi}(M)_{d^\ast}$. For this, let $\alpha^3 = d^\ast \alpha^4$ with $\alpha^4 \in d\Om^3(M)$. Then
\begin{align} \label{eq:dalpha4}
\alpha^3 \wedge \ast \varphi &= d^\ast \alpha^4 \wedge \ast \varphi = (\ast d \ast \alpha^4) \wedge \ast \varphi = d\ast \alpha^4 \wedge \varphi = d(\ast \alpha^4 \wedge \varphi),
\end{align}
whence
\[
\Ll_{\ast \varphi}(\alpha^3) = -d^\ast(\alpha^3 \wedge \ast \varphi) \stackrel{(\ref{eq:dalpha4})} = -d^\ast d(\ast \alpha^4 \wedge \varphi),
\]
so that $\Ll_{\ast \varphi}(\alpha^3) = 0$ iff $0 = d(\ast \alpha^4 \wedge \varphi) \stackrel{(\ref{eq:dalpha4})} = \alpha^3 \wedge \ast \varphi$. Moreover,
\begin{align*}
\Ll_{\ast \varphi}(\ast \alpha^3) = (d^\ast \ast \alpha^3) \wedge \ast \varphi = (\ast d \alpha^3) \wedge \ast \varphi = d \alpha^3 \wedge \varphi = d(\alpha^3 \wedge \varphi).
\end{align*}
Thus, $\alpha^3 \in \Hh^3_{\ast \varphi}(M)_{d^\ast}$ iff $\Ll_{\ast \varphi}(\alpha^3) = \Ll_{\ast \varphi}(\ast \alpha^3) = 0$ iff $\alpha^3 \wedge \ast \varphi = d(\alpha^3 \wedge \varphi) = 0$, and by (\ref{def-V3}) this is the case iff $\alpha^3 \in \Om^{3;d^\ast}_{27}(M) \oplus V_3^{d^\ast}(M)$.

Again, since $\ast: \Hh^l_{\ast \varphi}(M) \to \Hh^{7-l}_{\ast \varphi}(M)$ is an isomorphism, the assertions for $l = 4, 5$ follow as well using Proposition \ref{prop:diagrams}.

The $L^2$-orthogonality of the summands is straightforward by Lemma \ref{lem:V3} 
and Proposition \ref{prop:diagrams}.
\end{proof}

\section{The Fr\"olicher-Nijenhuis cohomology of $\Sp$-manifolds}\label{sec:spin7}

In this section, we repeat our discussion for the parallel $4$-form on a $\Sp$-manifold.

Let $W$ be an $8$-dimensional oriented real vector space. A {\em $\Sp$-structure on $W$ }is a form $\Phi \in \Lambda^4 W^\ast$ for which there is a positively oriented basis $(e_\mu)_{\mu=0}^7$ of $W$ such that
\begin{eqnarray} \label{Phi4}
\Phi & := & e^{0123} + e^{0145} + e^{0167} + e^{0246} - e^{0257} - e^{0347} - e^{0356}\\
\nonumber & & + e^{4567} + e^{2367} + e^{2345} + e^{1357} - e^{1346} - e^{1256} - e^{1247}.
\end{eqnarray}
Throughout this section, we shall use Greek indices $\mu, \nu, \ldots$ to run over $0, \ldots, 7$, whereas Latin indices $i,j, \ldots$ range over $1, \ldots, 7$.

If we define the forms $\varphi$ and $\ast_7 \varphi$ on $V := \rmspan(e_i)_{i=1}^7 \subset W$ as in (\ref{varphi}) and (\ref{varphi*}), then
\[
\Phi = e^0 \wedge \varphi + \ast_7 \varphi.
\]

The subgroup of $Gl (W)$ preserving $\Phi$ is isomorphic to $\Sp$, which acts irreducibly on $W$. Since $\Sp$ is compact, there is a unique $\Sp$-invariant inner product $g = g_\Phi$ on $W$ for which $\|\Phi\|^2_{g_\Phi} = 14$, and any positively oriented basis $(e_\mu)$ of $W$ for which (\ref{Phi4}) holds is orthonormal, whence $\Phi = \ast \Phi$ when taking the Hodge-$\ast$ w.r.t. $g_\Phi$.

As before, we denote by $W_k$ the $k$-dimensional irreducible $\Sp$-module if there is a unique such module. For instance, $W_8$ is the irreducible $8$-dimensional $\Sp$-module from above, and $W_8^\ast \cong W_8$ as there is a $\Sp$-invariant metric on $W_8$. For its exterior powers, we obtain the decompositions
\begin{align}
\nonumber
\Lambda^0 W_8 &\cong \Lambda^8 W_8 \cong W_1, \qquad
& \Lambda^2 W_8  \cong \Lambda^6 W_8 &\cong W_7 \oplus  W_{21},\\
\label{decom1-L-W8}
\Lambda^1 W_8 &\cong \Lambda^7 W_8 \cong W_8, \qquad
& \Lambda^3 W_8 \cong \Lambda^5 W_8 &\cong W_8 \oplus W_{48},\\
\nonumber
\Lambda^4 W_8 &\cong W_1 \oplus W_7 \oplus W_{27} \oplus W_{35}
\end{align}
where $\Lambda^k W_8 \cong \Lambda^{8-k} W_8$ due to $\Sp$-invariance of the Hodge-$\ast$ isomorphism $\ast: \Lambda^k W_8 \to \Lambda^{8-k} W_8^\ast \cong \Lambda^{8-k} W_8$. Again, we denote by $\Lambda^k_l W_8 \subset \Lambda^k W_8$ the subspace isomorphic to $W_l$ in the above notation. For instance,
\[
\Lambda^3_8 W_8 = \{ \imath_v \Phi \mid v \in W_8\} \quad \mbox{and} \quad \Lambda^3_{48} W_8 = \{ \alpha^3 \in \Lambda^3 W_8 \mid \alpha^3 \wedge \Phi = 0\}.
\]

A {\em $\Sp$-manifold }is a pair $(M, \Phi)$ consisting of an $8$-dimensional oriented Riemannian manifold with a $4$-form $\Phi \in \Om^4(M)$ such that at each $p \in M$, $\Phi_p$ can be written as in (\ref{Phi4}) w.r.t. some positively oriented basis $(e_\mu)$ of $T_pM$ with dual basis $(e^\mu)$ of $T_p^\ast M$, and such that $\Phi^4$ is parallel w.r.t. the Riemannian metric $g = g_\Phi$ on $M$. Moreover, (\ref{decom1-L-W8}) induces a decomposition of differential forms and canonical projections by
\[
\Om^k_l(M) := \Gamma(M, \Lambda^k_l T^\ast M), \qquad \pi^k_l: \Om^k(M) \longrightarrow \Om^k_l(M).
\]

In analogy to (\ref{eq:def-Om-kd}), we define the spaces
\begin{equation} \label{eq:DiffForm-Spin7}
\begin{array}{lll}
\Om^{k;d}_l(M) & = & \Om^k_l(M) \cap d \Om^{k-1}(M)\\[2mm]
\Om^{k;d^\ast}_l(M) & = & \Om^k_l(M) \cap d^\ast \Om^{k+1}(M),
\end{array}
\end{equation}
and by the same argument as in (\ref{eq:Om-kd-*}) and (\ref{eq:Om-kd-Laplace}), it is immediate that
\begin{equation} \label{eq:Om-kd-Laplace-Spin}
\triangle \Om^{k;d}_l(M) = \Om^{k;d}_l(M) \quad \mbox{and} \quad \triangle \Om^{k;d^\ast}_l(M) = \Om^{k;d^\ast}_l(M)
\end{equation}
and that the Hodge-$\ast$ yields isomorphisms $\ast: \Om^{k;d}_l(M) \xlongleftrightarrow{\ast} \Om^{8-k;d^\ast}_l(M)$.

In the proofs of the following results we shall again utilize the results presented in the Appendix \ref{sec:appendix}.

\begin{lemma} \label{lem:decompOm3}
Let $(M, \Phi)$ be a closed $\Sp$-manifold. Then there is a decomposition
\[
\Om^3_{48}(M) = \Hh^3_{48}(M) \oplus \Om^{3;d^\ast}_{48}(M) \oplus \pi^3_{48}(d\Om^3(M)),
\]
and $\Om^{3;d^\ast}_{48}(M)$ is infinite dimensional.
\end{lemma}

\begin{proof}
In complete analogy with the proof of Proposition \ref{prop:inf-dim}, the asserted decomposition will follow if we can show that the differential operator $\phi: \Om^2(M) \to \Om^3_{48}(M) \oplus \Om^1(M)$, $\alpha^2 \mapsto \pi^3_{48}(d\alpha^2) + d^\ast \alpha^2$ is overdetermined 
elliptic. The symbol of $\phi$ at $0 \neq \xi \in T_p^\ast M$ is given by
\[
\sigma_\xi(\alpha^2) = \pi^3_{48}(\xi \wedge \alpha^2) - \imath_{\xi^\#} \alpha^2,
\]
and after rescaling $\xi$ to a unit vector and using the fact that $\Sp$ acts transitively on the unit sphere, we may assume w.l.o.g. that $\xi = e^0$ in an orthonormal basis $(e_\mu)$ of $T_pM$ in which (\ref{Phi4}) holds. If $\alpha^2 \in \ker \sigma_{e^0}$, then $\pi^3_{48}(e^0 \wedge \alpha^2) = 0$, whence
\[
e^0 \wedge \alpha^2 = \imath_{c e_0 + v} \Phi = c \varphi - e^0 \wedge (\imath_v \varphi) + \imath_v \ast_7 \varphi
\]
for some $c \in \R$ and $v \in e_0^\perp$. Thus, $c \varphi + \imath_v \ast_7 \varphi = 0$, and from this, it easily follows that $c = 0$ and $v = 0$, so that $e^0 \wedge \alpha^2 = 0$. This together with $\imath_{e_0} \alpha^2 = 0$ implies that $\alpha^2 = 0$, so that $\ker \sigma_{e_0} = 0$, showing the claim.

The assertion on the infinite dimension of $\Om^{3;d^\ast}_{48}(M)$ follows when applying Lemma \ref{lem:codim-inf} to the differential operator $\Om^2(M) \ni \alpha^2 \mapsto \pi^3_{48}(d\alpha^2) \in \Om^3_{48}(M)$, as $rank(\Lambda^2 T^\ast M) = 28 < 48$, and using $\dim \Hh^3_{48}(M) < \infty$.
\end{proof}

\begin{theorem} \label{thm:cohomSpin7}
Let $(M, \Phi)$ be a closed $\Sp$-manifold. Then the differential $\Ll_\Phi$ is regular, and the cohomology algebra $H^\ast_\Phi(M)$ is given as follows:
\[
\begin{array}{rrrrrrrr}
H_\Phi^3(M) \cong & \Hh^3(M) & \oplus & \Om^{3;d^\ast}_{48}(M)\\[1mm]
H_\Phi^4(M) \cong & \Hh^4(M) & \oplus & d\Om^{3;d^\ast}_{48}(M) & \oplus & d^\ast \Om^{5;d}_{48}(M)\\[1mm]
H_\Phi^5(M) \cong & \Hh^5(M) & & & \oplus & \Om^{5;d}_{48}(M)\\[1mm]
\end{array}
\]
with the definitions in (\ref{eq:DiffForm-Spin7}), and $H_\Phi^k(M) \cong \Hh^k(M)$ for $k = 0,1,2,6,7,8$.\end{theorem}

Observe that by Lemma \ref{lem:decompOm3}, all summands in this decomposition apart from $\Hh^k(M)$ are infinite dimensional.

\begin{proof} We begin by showing the regularity of $\Ll_{\Phi;l}: \Om^{l-3}(M) \to \Om^l(M)$ by showing that they are all either underdetermined elliptic or overdetermined elliptic. In order to see this, we first show that the $G_2$-equivariant map
\[
\wedge \varphi : \Lambda^{3-l} V_7^\ast \longrightarrow \Lambda^l V_7^\ast
\]
is injective for $l \leq 5$ and surjective for $l \geq 5$. 

For $l = 3$, let $0 \neq c \in \R = \Lambda^0V_7^\ast$. Then $c \wedge \varphi = c \varphi \neq 0$, showing the injectivity. Likewise, for $l = 4$, $\alpha^1 \wedge \varphi \neq 0$ for $0 \neq \alpha^1 \in V_7^\ast$ is immediate. 

For $l = 5$, use the decomposition $\Lambda^2 V_7^\ast = \Lambda^2_7 V_7^\ast \oplus \Lambda^2_{14} V_7^\ast$, and then Lemma \ref{Lambda2-V7} immediately implies that the map $\wedge \varphi: \Lambda^2 V_7^\ast \to \Lambda^5 V_7^\ast$ is an isomorphism, whence both injective and surjective.

For $l = 6$, $(\imath_v \ast \varphi) \wedge \varphi \stackrel{(\ref{eq:form16})}= 4 \ast v^\flat$ which shows the surjectivity of $\wedge \varphi: \Lambda^3 V_7^\ast \to \Lambda^6 V_7^\ast$, and for $l=7$, $\wedge \varphi: \Lambda^4 V_7^\ast \to \Lambda^7 V_7^\ast = \R\; \vol$ is surjective since $\ast \varphi \wedge \varphi = \|\varphi\|^2 \vol \neq 0$.

Let $\xi \in T_p^\ast M$ be a unit vector. Since $\Sp$ acts transitively on the unit sphere, we may choose an orthonormal basis $(e_\mu)$ of $T_pM$ for which (\ref{Phi4}) holds such that $\xi = e^0$. Thus, by (\ref{eq:symbol of L_K}) the symbol $\sigma_{e^0}(\Ll_{\Phi;l}): \Lambda^{l-3}T_p^\ast M \to \Lambda^l T_p^\ast M$ of $\Ll_{\Phi}|_{\Om^{l-3}(M)}: \Om^{l-3}(M) \rightarrow \Om^l(M)$ is given w.r.t. this basis as
$$
\sigma_{e_0}(\Ll_{\Phi;l}): \Lambda^{l-3} W_8^\ast \longrightarrow \Lambda^l W_8^\ast, \qquad \alpha^{l-3} \longmapsto \varphi \wedge \alpha^{l-3}.
$$
As before, let $V_7 := e_0^\perp \subset W_8$. Evidently, the splitting $\Lambda^k W_8^\ast =  \Lambda^k V_7^\ast \oplus e^0 \wedge \Lambda^{k-1} V_7^\ast$ is preserved by $\sigma_{e_0}(\Ll_{\Phi;l})$, and from the above, it now follows that $\sigma_{e_0}(\Ll_{\Phi;l})$ is injective for $l \leq 5$ and surjective for $l \geq 6$.

Thus, $\Ll_\Phi$ is regular, so that by Theorem \ref{thm:cohom-harmonic} (1), $H^l_\Phi(M) \cong \Hh^l_\Phi(M)$.

Since $\Phi$ is multi-symplectic, Proposition \ref{prop:chomom-LK_0,n} implies that $H^l_{\Phi}(\Om^\ast(M)) = \Hh^l(M)$ for $l = 0,8$. Moreover, Proposition \ref{prop:chomom-LK 1-form} implies that $\alpha^1 \in H^1_{\Phi}(\Om^\ast(M)) = \ker \Ll_{\Phi}|_{\Om^1(M)}$ only if $\Ll_{(\alpha^1)^\#} (\Phi) =0$. But every vector field whose flow preserves $\Phi$ must be a Killing field, and since $\Sp$-manifolds are Ricci flat, Bochner's theorem implies that $(\alpha^1)^\#$ is parallel, whence so is $\alpha^1$. In particular, $\alpha^1$ is harmonic, and this shows that $H^1_{\Phi}(\Om^\ast(M)) = \Hh^1(M)$, and $H^7_{\Phi}(\Om^\ast(M)) = \ast H^1_{\Phi}(\Om^\ast(M)) = \ast \Hh^1(M) = \Hh^7(M)$.

Next, for $l = 2$, we have $\Hh^2_\Phi(M)_d = 0$ by (\ref{eq:Hl-split1}). Thus, we need to determine
\[
\Hh^2_\Phi(M)_{d^\ast} = \{ \alpha^2 \in d^\ast \Om^3(M) \mid d^\ast (\alpha^2 \wedge \Phi) = 0\}.
\]
By (\ref{decom1-L-W8}), we may decompose $\alpha^2 = \alpha^2_7 + \alpha^2_{21}$, where $\alpha_j \in \Om^2_j (M)$.
By Lemma \ref{Kar47-48} we have 
\begin{equation} \label{wedge-Phi}
\alpha^2_7 \wedge \Phi = 3 \ast \alpha^2_7, \qquad 
\alpha^2_{21} \wedge \Phi = -\ast \alpha^2_{21}.
\end{equation}
Thus, 
\begin{align*}
d^\ast(\alpha^2 \wedge \Phi) & = \ast d \ast (\alpha^2 \wedge \Phi) =
\ast d (3 \alpha^2_7 - \alpha^2_{21}),\\
3 d^\ast \alpha^2 & = 3 \ast d \ast \alpha^2 = \ast (d(\alpha^2_7 \wedge \Phi -3 \alpha^2_{21} \wedge \Phi))
= \ast ((d \alpha^2_7 -3 d \alpha^2_{21}) \wedge \Phi).
\end{align*}
Thus 
$\Ll_{\Phi} \alpha^2 = 0$ iff 
$d (3 \alpha^2_7 - \alpha^2_{21}) = 0$ and $(d \alpha^2_7 -3 d \alpha^2_{21}) \wedge \Phi = 0$. 
Substituting $d\alpha^2_{21} = 3 d\alpha^2_7$ from the first into the second equation, it follows that
\begin{align*}
0 &= d\alpha^2_7 \wedge \Phi = d(\alpha^2_7 \wedge \Phi) \stackrel{(\ref{wedge-Phi})} = 3 d \ast \alpha^2_7\\
0 &= d\alpha^2_{21} \wedge \Phi = d(\alpha^2_{21} \wedge \Phi) \stackrel{(\ref{wedge-Phi})} = - d \ast \alpha^2_{21},
\end{align*}
so that $d^\ast \alpha^2_7 = d^\ast \alpha^2_{21} = 0$. Thus,
$$
d(3 \alpha^2_7 - \alpha^2_{21}) = 0 \qquad \mbox{and} \qquad d^\ast (3 \alpha^2_7 - \alpha^2_{21}) = 0,
$$
whence $3 \alpha^2_7 - \alpha^2_{21}$ is harmonic. 
Since the Laplacian preserves the irreducible decomposition by the Weitzenb\"ock formula (see e.g. \cite[(1.154)]{Besse1987}), it follows that $\alpha^2_7$ and $\alpha^2_{21}$ are harmonic, whence so is $\alpha^2$. Thus, $\Hh^2_{\Phi}(M) \cong \Hh^2(M)$ and hence, $\Hh^6_{\Phi}(M) \cong \Hh^6(M)$.

Now let $l = 3$. Since $\Hh^2_{\Phi}(M) = \Hh^2(M)$, it follows from (\ref{eq:Hl-split}) that $\Hh^3_{\Phi}(M)_d = 0$, whence by (\ref{eq:decomp-Hh_K}) we only need to determine
\[
\Hh^3_{\Phi}(M)_{d^\ast} = \{ \alpha^3 \in d^\ast \Om^4(M) \mid d^\ast(\alpha^3 \wedge \Phi) = 0, (d^\ast \ast \alpha^3) \wedge \Phi = 0\}.
\]
Observe that
\[
(d^\ast \ast \alpha^3) \wedge \Phi = (\ast d \alpha^3) \wedge \Phi \stackrel{\Phi = \ast \Phi} = (d \alpha^3) \wedge \Phi = d(\alpha^3 \wedge \Phi)
\]
whence $\alpha^3 \in \Hh^3_{\Phi}(M)_{d^\ast}$ iff $\alpha^3 \in d^\ast \Om^4(M)$ and $\alpha^3 \wedge \Phi \in \Hh^7(M)$.

On the other hand, for $\alpha^3 \in d^\ast \Om^4(M) \subset \triangle \Om^4(M)$, (\ref{eq:Laplace-commute}) implies that $\alpha^3 \wedge \Phi \in \triangle (\Om^7(M)) = (\Hh^7(M))^\perp$, whence
\[
\Hh^3_{\Phi}(M)_{d^\ast} = \{ \alpha^3 \in d^\ast \Om^4(M) \mid \alpha^3 \wedge \Phi = 0\} = \Om^{3;d^\ast}_{48}(M)
\]
and hence, $\Hh^3_{\Phi}(M)$ is of the asserted form.

Thus, $\Hh^5_{\Phi}(M) = \ast \Hh^3_{\Phi}(M) \cong \Hh^5(M) \oplus \Om^{5;d}_{48}(M)$, and by (\ref{eq:Hl-split}), $\Hh^4_{\Phi}(M)$ is of the asserted form.
\end{proof}

\section{Deformations of $G_2$-and $\Sp$-structures}\label{sec:deform}

While Theorems \ref{thm:cohomG2} and \ref{thm:cohomSpin7} give a complete description of $H^\ast_{\ast \varphi}(M^7)$ and $H^\ast_\Phi(M^8)$ for closed $G_2$- and $\Sp$-manifolds, respectively, the cohomology Lie algebras $H^\ast_{\ast \varphi}(M^7, TM^7)$ and $H^\ast_\Phi(M^8, TM^8)$ seem to be more difficult to determine. In this section, we shall use Proposition \ref{prop:homology-general} to describe $H^k_{\ast \varphi}(M, TM)$ and $H^k_\Phi(M, TM)$, respectively, for $k = 0$ and $k = 3$.

Let $V$ be an oriented $7$-dimensional vector space, and let $\Lambda^3 _{G_2} V^\ast \subset \Lambda^3 V^\ast$ be the set of $3$-forms which can be written as in (\ref{varphi}) for some oriented basis $(e^i)$ of $V^\ast$. Likewise, for an oriented $8$-dimensional vector space $W$, we let $\Lambda^4_\Sp W^\ast\subset \Lambda^4 W^\ast$ be the set of $4$-forms which can be written as in (\ref{Phi4}) for some oriented basis $(e^\mu)$ of $W^\ast$. By definition, the groups $Gl^+(V)$ and $Gl^+(W)$ of orientation preserving automorphisms of $V$ and $W$, respectively, act transitively on $\Lambda^3 _{G_2} V^\ast$ and $\Lambda^4_\Sp W^\ast$, respectively, so that
\[
\Lambda^3 _{G_2} V^\ast = Gl^+(V)/G_2, \qquad \mbox{and} \qquad 
\Lambda^4_\Sp W^\ast = Gl^+(W)/\Sp.
\]

\begin{lemma}\label{lem:uniq} With $V$ and $W$ as above, the maps
\begin{eqnarray*} & \frak{C}: \Lambda^3 _{G_2} V^\ast \longrightarrow \Lambda ^3 V^\ast \otimes V, & \qquad \varphi \longmapsto \p_{g_\varphi} (\ast_{g_\varphi} \varphi)\\
& \frak{P}: \Lambda^4_\Sp W^\ast \longrightarrow \Lambda^3 W^\ast \otimes W, & \qquad \Phi \longmapsto \p_{g_\Phi} \Phi
\end{eqnarray*}
are injective immersions, and they are $Gl^+(V)$- and $Gl^+(W)$-equivariant, respectively. Here, $\p_g$ is the map from (\ref{eq:def-partial}), and where as before, $g_\varphi$ and $g_\Phi$ denotes the metric induced by $\varphi$ and $\Phi$, respectively.
\end{lemma}

\begin{proof} The equivariance of these maps is immediate from the definition. Thus, as $\Lambda^3 _{G_2} V^\ast $ and $\Lambda^4_\Sp W^\ast$ are homogeneous spaces, $\frak{C}$ and $\frak{P}$ are injective immersions iff $Stab_{Gl^+(7,\R)}(\p_g (\ast_{g_\varphi} \varphi)) = G_2$ and $Stab_{Gl^+(8,\R)}(\p_g \Phi) = \Sp$, respectively. The inclusions $\supseteq$ are evident. For the reverse inclusion, observe that the only intermediate Lie groups are given by the following inclusions for any subgroup $\Gamma \subset \R^+ Id$ (see e.g. \cite{Dynkin1952}):

\begin{equation} \label{intermediate}
\xymatrix{
\Gamma \cdot G_2 \ar@{^(->}[r] & \Gamma \cdot SO(7) \ar@{^(->}[r] & \Gamma \cdot Sl^+(7, \R)\\
G_2 \ar@{^(->}[r] \ar@{^(->}[u] & SO(7) \ar@{^(->}[r] \ar@{^(->}[u] & Sl(7, \R) \ar@{^(->}[u]\\
\Gamma \cdot \Sp \ar@{^(->}[r] & \Gamma \cdot SO(8) \ar@{(->}[r] & \Gamma \cdot Sl^+(8, \R)\\
\Sp \ar@{^(->}[r] \ar@{^(->}[u] & SO(8) \ar@{^(->}[r] \ar@{^(->}[u] & Sl(8, \R) \ar@{^(->}[u]
}
\end{equation}
Since $\Lambda^3 V^\ast$ and $V^\ast$ ($\Lambda^3 W^\ast$ and $W^\ast$, respectively) are inequivalent irreducible $SO(7)$-modules ($SO(8)$-modules, respectively), there is no non-zero element in $\Lambda^3 V^\ast \otimes V$ (in $\Lambda^3 W^\ast \otimes W$, respectively) which is invariant under $SO(7)$ ($SO(8)$, respectively). Thus, $Stab_{Gl^+(V)} \p_g (\ast_{g_\varphi} \varphi)$ cannot contain $SO(7)$ and $Stab_{Gl^+(V)} \p_g \Phi$ cannot contain $SO(8)$.

Likewise, for $\lambda > 0$, we have $(\lambda Id_V) \cdot \p_g (\ast_{g_\varphi} \varphi) = \lambda^{-4} \p_g (\ast_{g_\varphi} \varphi)$ and $(\lambda Id_V) \cdot \p_g \Phi = \lambda^{-4} \p_g \Phi$, whence $Stab_{Gl^+(V)} \p_g (\ast_{g_\varphi} \varphi) \cap \R^+ Id_V = Id_V$ and $Stab_{Gl^+(V)} \p_g \Phi \cap \R^+ Id_W = Id_W$, and from this and (\ref{intermediate}) it follows that $Stab_{Gl^+(7,\R)}(\p_g (\ast_{g_\varphi} \varphi)) = G_2$ and $Stab_{Gl^+(8,\R)}(\p_g \Phi) = \Sp$, which completes the proof.
\end{proof}

\begin{proposition} \label{prop:H0MTM}
For a $G_2$-manifold $(M^7, \varphi)$,
\begin{equation} \label{H0MTM_G2}
H^0_{\ast \varphi}(M^7, TM^7) = \{ X \in {\frak X}(M^7) \mid \Ll_X \varphi = 0\}.
\end{equation}
In particular, if $M^7$ is closed, then $\dim H^0_{\ast \varphi}(M^7, TM^7) = b^1(M^7)$.

Likewise, for a $\Sp$-manifold $(M^8, \Phi)$,
\begin{equation} \label{H0MTM_Sp}
H^0_\Phi(M^8, TM^8) = \{ X \in {\frak X}(M^8) \mid \Ll_X \Phi = 0\}.
\end{equation}
In particular, if $M^8$ is closed, then $\dim H^0_\Phi(M^8, TM^8) = b^1(M^8)$.
\end{proposition}

This proposition implies the 4th parts in Theorems \ref{thm:hG2} and \ref{thm:hSp}, respectively.

For a closed $G_2$- or $\Sp$-manifold, $b^1(M) = 0$ unless the holonomy of $(M, g)$ is contained in $SU(3) \subsetneq G_2$ or in $G_2 \subsetneq \Sp$, respectively. Thus, for a closed $G_2$- or $\Sp$-manifold with full holonomy, the $0$-order cohomology vanishes.

\begin{proof}
Let $X \in {\frak X}(M^7)$ be a vector field, $p \in M^7$ and denote by $F_X^t$ the local flow along $X$, defined in a neighborhood of $p$. Then because of the pointwise equivariance of $\frak{C}$ we have
\[
(F_X^t)^\ast \Big(\p_ g \ast \varphi \Big)_{F_X^t(p)} = (F_X^t)^\ast \Big(\frak{C}(\varphi)_{F_X^t(p)}\Big) = \frak{C}\Big((F_X^t)^\ast(\varphi_{F_X^t(p)})\Big)
\]
and taking the derivative at $t=0$ yields
\begin{equation} \label{eq:Lie-equiv}
\Ll_X (\p_g \ast \varphi)_p = \Ll_X (\frak{C}(\varphi))_p = d\frak{C} (\Ll_X \varphi)_p.
\end{equation}
Now $\Ll_X (\p_g \ast \varphi) = [X, \p_g \ast \varphi]^{FN}$, and since $\frak C$ is an immersion by Lemma \ref{lem:uniq}, it follows that $X \in H^0_{\ast \varphi}(M^7, TM^7) = \ker ad_{\p_g \ast \varphi}$ iff $\Ll_X\varphi = 0$, showing (\ref{H0MTM_G2}).

Since $\varphi$ uniquely determines the Riemannian metric $g_\varphi$ on $M^7$, any vector field satisfying $\Ll_X \varphi = 0$ must be a Killing vector field, and if $M^7$ is closed, the Ricci flatness of $G_2$-manifolds and Bochner's theorem imply that $X$ must be parallel, showing that in this case, $\dim H^0_\Phi(M^7, TM^7) = b^1(M^7)$.

The proof for $\Sp$-manifolds is completely analogous.
\end{proof}

Let us now consider deformations of $G_2$- and $\Sp$-structures. Recall that a $G_2$-structure on an oriented manifold $M^7$ is a $3$-form $\varphi \in \Om^3(M^7)$ such that at each $p \in M^7$, $\varphi_p$ is of the form (\ref{varphi}) for some oriented basis $(e^i)$ of $T_p^\ast M^7$. Similarly, a $\Sp$-structure, on an oriented manifold $M^8$ is a $4$-form $\Phi \in \Om^4(M^8)$ such that at each $p \in M^8$, $\Phi_p$ is of the form (\ref{Phi4}) for some oriented basis $(e^\mu)$ of $T_p^\ast M^8$.

Evidently, a $G_2$- and $\Sp$-structure induces a Riemannian metric on the underlying manifold, but in general, $\varphi$ and $\Phi$, respectively, need not be parallel w.r.t. this metric.

\begin{definition} \label{def:deform}
Let $(M^7, \varphi_0)$ be a $G_2$-manifold.  A 3-form $\dot \varphi_0 \in \Om^3(M^7)$ is called a {\em torsion free infinitesimal deformation of $\varphi_0$} if there exists a family $(\varphi_t)$ of $G_2$-structures which depend fiberwise smoothly on $t$, such that 
\begin{equation} \label{eq:infi-def-G2}
\left. \varphi_t \right|_{t=0} = \varphi_0, \quad 
\left. \dfrac d{dt} \right|_{t=0} \varphi_t = \dot \varphi_0, \quad \mbox{and} \quad \left. \dfrac d{dt} \right|_{t=0} \nabla^{g_t} \varphi_t = 0,
\end{equation}
where $g_t = g_{\varphi_t}$ is the Riemannian metric induced by $\varphi_t$ and where the derivatives are taken fiberwise.
We call an infinitesimal deformation {\em trivial }if $\dot \varphi_0 = \Ll_X \varphi_0$ for some vector field $X \in {\frak X}(M^7)$.

For a $\Sp$-manifold $(M^8, \Phi_0)$, torsion free infinitesimal deformations and trivial deformations of $\Phi_0$ are defined analogously.
\end{definition}

For a $G_2$-structure $\varphi$ and the induced Riemannian metric $g_\varphi$, there is a section $T \in \Om^1(M^7, TM^7)$, called the {\em torsion endomorphism of $\varphi$}, (cf. \cite[Proposition 3.3]{KLS}, see also \cite{FG1982}) such that for $v \in TM^7$
\begin{equation} \label{eq:nabla-G2}
\nabla_v \varphi = \imath_{T(v)} \ast_{g_\varphi} \varphi.
\end{equation}

Thus, for $v \in TM$ we have
\[
\left. \dfrac d{dt} \right|_{t=0} \nabla^{g_t}_v \varphi_t = \imath_{(\left.\frac d{dt}\right|_0 T^t(v))} (\ast_{g_0} \varphi_0),
\]
whence the last equation in (\ref{eq:infi-def-G2}) is equivalent to
\begin{equation} \label{eq:dT=0}
\left. \dfrac d{dt} \right|_{t=0} T^t = 0,
\end{equation}
again taking the derivative pointwise. 

It was shown in \cite[Section 5]{KLS} that there is a bundle isomorphism $\tau: T^\ast M^7 \otimes TM^7 \to \Lambda^6 T^\ast M^7 \otimes TM^7$ such that for $\chi_\varphi := \frak{C}(\varphi) = \p_{g_\varphi} \ast_{g_\varphi} \varphi \in \Om^3(M^7, TM^7)$ we have
\begin{equation} \label{prop:chi-chi}
{}[\chi_\varphi, \chi_\varphi]^{FN} = \tau(T) \in \Om^6(M, TM),
\end{equation}
where by abuse of notation we denote by $\tau: \Om^1(M, TM) \to \Om^6(M, TM)$ the pointwise application of $\tau$ to sections. Therefore,
\begin{align*}
\tau \left(  \left. \dfrac d{dt} \right|_{t=0} T^t \right) = \left. \dfrac d{dt} \right|_{t=0} \tau(T^t) &= \left. \dfrac d{dt} \right|_{t=0} [\chi_{\varphi_t}, \chi_{\varphi_t}] = 2 \left[ \chi_{\varphi_0}, \left. \dfrac d{dt} \right|_{t=0}\chi_{\varphi_t} \right]\\
&= 2 \left[ \chi_{\varphi_0}, d{\frak C}(\dot \varphi_0)\right],
\end{align*}
so that $\dot \varphi_0  \in \Om^3(M^7)$ is a torsion free infinitesimal deformation of $\varphi_0$ iff (\ref{eq:dT=0}) holds iff $d{\frak C}(\dot \varphi_0) \in \ker(ad_{\chi_{\varphi_0}}: \Om^3(M^7, TM^7) \to \Om^6(M^7, TM^7))$. Since ${\frak C}$ is an immersion and hence $d{\frak C}$ injective by Lemma \ref{lem:uniq}, we have an isomorphism
\begin{align*}
\{ \text{torsion free\;}&\text{infinitesimal deformations of $\varphi_0$}\} \stackrel{d{\frak C}} \cong\\
\nonumber
&\ker\Big(ad_{\chi_{\varphi_0}}: \Om^3(M^7, TM^7) \to \Om^6(M^7, TM^7)\Big) \cap \Im(d{\frak C}).
\end{align*}

Observe that by (\ref{eq:Lie-equiv})
\begin{equation*}
d{\frak C}(\Ll_X \varphi_0) = \Ll_X({\frak C}(\varphi_0)) = [X, \chi_{\varphi_0}]^{FN} = - ad_{\chi_{\varphi_0}}(X),
\end{equation*}
whence there is an induced inclusion
\begin{align*}
\dfrac{\{ \text{torsion free infinitesimal deformations of $\varphi_0$}\}}{\{\text{trivial deformations of $\varphi_0$}\}}\xhookrightarrow{\quad \mbox{$d\frak C$}\quad} H^3_{\varphi_0}(M^7, TM^7).
\end{align*}

The deformations of a torsion free $\Sp$-structure $\Phi_0$ on an oriented manifold $M^8$ can be discussed analogously. Here, we get an isomorphism
\begin{align*}
\{ \text{torsion free\;}&\text{infinitesimal deformations of $\Phi_0$}\} \stackrel{d{\frak P}} \cong\\
\nonumber
&\ker\Big(ad_{P_{\Phi_0}}: \Om^3(M^8, TM^8) \to \Om^6(M^8, TM^8)\Big) \cap \Im(d{\frak P}),
\end{align*}
and for a vector field $X \in {\frak X}(M^8)$ we have
\begin{equation*} 
d_{\varphi_0}{\frak P}(\Ll_X \varphi_0) = \Ll_X({\frak P}(\Phi_0)) = [X, P_{\Phi_0}]^{FN} = - ad_{P_{\Phi_0}}(X)
\end{equation*}
and hence an injective map
\begin{align*}
\dfrac{\{ \text{torsion free infinitesimal deformations of $\Phi_0$}\}}{\{\text{trivial deformations of $\Phi_0$}\}} \hookrightarrow H^3_{\Phi_0}(M^8, TM^8).
\end{align*}

Thus, what we have shown is the following proposition, which implies the 5th parts in Theorems \ref{thm:hG2} and \ref{thm:hSp}, respectively.

\begin{proposition} \label{prop:H3MTM}
For a $G_2$-manifold $(M^7, \varphi)$, the cohomology $H^3_{\ast \varphi}(M^7, TM^7)$ contains the space of all torsion free infinitesimal deformations of $\varphi$ modulo trivial deformations. In particular, if $M^7$ is closed then $\dim H^3_{\ast \varphi}(M^7, TM^7) \geq b^3(M^7) >0$.

Likewise, for a $\Sp$-manifold $(M^8, \Phi)$, the cohomology $H^3_\Phi(M^8, TM^8)$ contains the space of all torsion free infinitesimal deformations of $\Phi$ modulo trivial deformations. In particular, if $M^8$ is closed then $\dim H^3_\Phi(M^8, TM^8) \geq b^4_1(M^8)+b^4_7(M^8)+b^4_{35}(M^8) >0$ (cf. (\ref{eq:bkl}) for the definition of $b^4_l(M)$).
\end{proposition}

The final statements on the relation of deformations and the Betti numbers is due to the regularity theorem of the moduli space of such manifolds due to D. Joyce, see \cite{JoyceG2}, \cite[Theorem 11.2.8]{Joyce2007} for the case of $G_2$ and \cite{JoyceSp}, \cite[Theorem 11.5.9]{Joyce2007} for $\Sp$-manifolds.

\section{Appendix} \label{sec:appendix}

In this appendix, we shall collect some formulas and results which we need in the calculations in this paper. 

\begin{lemma} \label{formulas} (cf. \cite[Lemma 6.1]{KLS}) For all $u \in V_7$ the following identities hold.
\begin{align}
\label{eq:form15}
\ast \varphi \wedge (\imath_u \varphi) =\;& 3 \ast u^\flat\\
\label{eq:form16}
\ast (u^\flat \wedge \varphi) \wedge \varphi =\; \varphi \wedge (\imath_u \ast \varphi) =\;& -4 \ast u^\flat
\end{align}
\end{lemma}

We also get the following decomposition of $\Lambda^2 V_7$.

\begin{lemma} \label{Lambda2-V7} (cf. \cite[Lemma 6.2]{KLS})
Decompose $\Lambda^2 V_7^\ast = \Lambda^2_7 V_7^\ast \oplus \Lambda^2_{14} V_7^\ast$ according to (\ref{decom-L-V7}). Then
\begin{align*}
\Lambda^2_7 V_7^\ast &= \{ \alpha^2 \in \Lambda^2 V_7^\ast \mid \ast(\alpha^2 \wedge \varphi) = 2 \alpha^2\}, \quad \mbox{and}\\
\Lambda^2_{14} V_7^\ast &= \{ \alpha^2 \in \Lambda^2 V_7^\ast \mid \ast(\alpha^2 \wedge \varphi) = - \alpha^2\}.
\end{align*}
In particular,
\begin{equation} \label{eq:describe-Lamb2V7}
\begin{array}{lll}
\Lambda^2_7 V_7^\ast &= &\{ \alpha^2 + \ast(\alpha^2 \wedge \varphi) \mid \alpha^2 \in \Lambda^2 V_7^ \ast\}, \quad \mbox{and}\\[2mm]
\Lambda^2_{14} V_7^\ast &= &\{ 2 \alpha^2 - \ast(\alpha^2 \wedge \varphi) \mid \alpha^2 \in \Lambda^2 V_7^\ast\}.
\end{array}
\end{equation}
\end{lemma}

The following describes identities of representation of $\Sp$.

\begin{lemma} \label{Kar47-48}
(cf. \cite[(4.7), (4.8)]{Kar2005}) Decompose $\Lambda^2 W_8^\ast = \Lambda^2_7 W_8^\ast \oplus \Lambda^2_{21} W_8^\ast$ according to (\ref{decom1-L-W8}). Then
\begin{align*}
\Lambda^2_7 W_8^\ast &= \{ \alpha^2 \in \Lambda^2 W_8^\ast \mid \Phi \wedge \alpha^2 = 3 \ast \alpha^2\}, \quad \mbox{and}\\
\Lambda^2_{21} W_8^\ast &= \{ \alpha^2 \in \Lambda^2 W_8^\ast \mid \Phi \wedge \alpha^2 = - \ast \alpha^2\}.
\end{align*}
\end{lemma}

Finally we also recall the following standard result from the theory of differential operators.
\begin{lemma} \label{lem:codim-inf}
Let $E \to M$ and $F \to M$ be vector bundles, and $\phi: \Gamma(E) \to \Gamma(F)$ be a linear differential operator.

If $rank(E) < rank(F)$, then the image $\phi(\Gamma(E))$ has infinite codimension in $\Gamma(F)$.
\end{lemma}

\begin{proof}
Let $\phi: \Gamma(E) \to \Gamma(F)$ be a linear differential operator of degree $d$, and fix $p \in M$. For $N \in \N$, the $N$-th Taylor polynomial at $p$ of a section in $E$ and $F$, respectively, is an element of $P_N (T_pM) \otimes E_p$ and $P_N (T_pM) \otimes F_p$, respectively, where
\[
P_N (T_pM) := \bigoplus_{k=0}^N \odot^k(T_p^\ast M)
\]
is the space of polynomials of degree at most $N$ on $T_pM$. The differential operator $\phi$ then induces for each $N$ a linear map
\[
\phi_N: P_{N+d}(T_pM) \otimes E_p \longrightarrow P_N(T_pM) \otimes F_p,
\]
associating to the $(N+d)$-th order Taylor polynomial of $\sigma \in \Gamma(E)$ the $N$-th order Taylor polynomial of $\phi(\sigma) \in \Gamma(F)$. Since any polynomial in $P_N(T_pM) \otimes F_p$ occurs as the Taylor polynomial of some section of $F$, it follows that for all $N \in \N$
\begin{align*}
\codim (\phi(\Gamma(E)) \subset \Gamma(F)) & \geq \codim (\phi_N(P_{N+d}(T_pM) \otimes E_p) \subset P_N(T_pM) \otimes F_p)\\
&\geq \dim P_N(T_pM) \otimes F_p - \dim P_{N+d}(T_pM) \otimes E_p.
\end{align*}
Let $n := \dim M$. Then
\[
\lim_{N \to \infty} \dfrac1{N^n} \dim P_N(T_pM) = \lim_{N \to \infty} \dfrac1{N^n} \binom{N+n}n = \dfrac1{n!},
\]
whence
\begin{align*}
\lim_{N \to \infty} & \dfrac1{N^n} \left(\dim P_N(T_pM) \otimes F_p - \dim P_{N+d}(T_pM) \otimes E_p\right)\\
& = \dfrac1{n!}(rank(F) - rank(E)) > 0,
\end{align*}
so that $\lim_{N \to \infty} \dim P_N(T_pM) \otimes F_p - \dim P_{N+d}(T_pM) \otimes E_p = \infty$, and this shows the claim.
\end{proof}

\subsection*{Acknowledgements} 
A part of this project has been  discussed  during HVL's visit to the Osaka City University in December 2015. She thanks Professor Ohnita for his invitation to Osaka and his hospitality. HVL and LS also thank the Max Planck Institute for Mathematics in the Sciences in Leipzig for its hospitality during extended visits. We also thank the referee for carefully reading the manuscript and for many valuable suggestions.


\end{document}